\documentclass[11pt, a4paper]{article}
\usepackage{afterpage}
\usepackage{algorithm}
\usepackage{algorithmicx}
\usepackage{algpseudocode}
 \usepackage{authblk}
\usepackage{amsmath}
\usepackage[title]{appendix}
\usepackage{amsfonts}
\usepackage{booktabs} 
\usepackage[font=small,labelfont=bf, hypcap=false]{caption} 
\usepackage{chngcntr}
\usepackage{comment} 
\usepackage{csquotes} 
\usepackage{csvsimple} 
\usepackage{diagbox}
\usepackage{enumitem}
\usepackage{eqnarray}
\usepackage{float} 
\usepackage[T1]{fontenc} 
\usepackage{geometry}
\usepackage{graphicx}
\usepackage[colorlinks=true, allcolors=blue]{hyperref}
\usepackage[none]{hyphenat} 
\usepackage{lineno} 
\usepackage{listings}
\usepackage{lipsum}
\usepackage{longtable}
\usepackage{lscape}
\usepackage{mathrsfs}
\usepackage{multirow}
\usepackage{orcidlink}
\usepackage{pdflscape}
\usepackage[]{ragged2e}
\usepackage{subcaption}
\usepackage{setspace} 
\usepackage{threeparttable} 
\usepackage{tabularray}
\usepackage{titlesec} 
\usepackage{cite}

\titleformat{\section} 
{\normalfont\Large\bfseries}{\thesection.}{1em}{}



\newcommand{\norm}[1]{\left\| #1 \right\|}

\graphicspath{{fig/}{fig/2d/}{fig/3d/}}

\usepackage{amssymb}
\usepackage{amsmath}
 \usepackage{amsthm} 
\usepackage{amsbsy}
\usepackage{color}
\newtheorem{theorem}{Theorem}[section]
\newtheorem{lemma}[theorem]{Lemma}
\newtheorem{definition}[theorem]{Definition}

\newtheorem{remark}[theorem]{Remark}

\numberwithin{equation}{section}

\usepackage{bm}
\usepackage{stmaryrd}
\usepackage{xparse}

\title{An FDM-sFEM scheme on time-space manifolds and its superconvergence analysis}

\author[1]{Chengrun Jiang}
\affil[1]{Department of Mathematical Sciences, Tsinghua University, Beijing, 100084,  China (\href{mailto:jcr22@mails.tsinghua.edu.cn}{jcr22@mails.tsinghua.edu.cn})}

\author[2]{Guozhi Dong}
\affil[2]{School of Mathematics and Statistics, HNP-LAMA,  Central South University, Changsha 410083, China (\href{mailto:guozhi.dong@csu.edu.cn}{guozhi.dong@csu.edu.cn})}

\author[3]{Hailong Guo}
\affil[3]{School of Mathematics and Statistics,  The University of Melbourne,  Parkville, VIC 3010, Australia (\href{mailto:hailong.guo@unimelb.edu.au}{hailong.guo@unimelb.edu.au})}

\author[4,5]{Zuoqiang Shi}
\affil[4]{Yau Mathematical Sciences Center, Tsinghua University, Beijing, 100084, China (\href{mailto:zqshi@tsinghua.edu.cn}{zqshi@tsinghua.edu.cn})}
\affil[5]{Yanqi Lake Beijing Institute of Mathematical Sciences and Applications, Beijing, 101408, China}

\date{}

\begin{document}
\newpage
\maketitle
\begin{abstract}
We study superconvergent discretization of the Laplace–Beltrami operator on time-space product manifolds with Neumann temporal boundary values, which arise in the context of dynamic optimal transport on general surfaces. We propose a coupled scheme that combines finite difference methods in time with surface finite element methods in space. By establishing a new summation by parts formula and proving the supercloseness of the semi-discrete solution, we derive superconvergence results for the recovered gradient via post-processing techniques. In addition, our geometric error analysis is implemented within a novel framework based on the approximation of the Riemannian metric. Several numerical examples are provided to validate and illustrate the theoretical results.

\end{abstract}



\section{Introduction}\label{sec:int}
Let $(\mathcal{M},g)$ be a $2$-dimensional connected, oriented, compact, and $C^3$-smooth Riemannian manifold, where $g$ denotes the Riemannian metric tensor. In this work, we consider the Laplace--Beltrami type equation with Neumann boundary condition posed on the product manifold $\mathcal{T} \times \mathcal{M}$ with $\mathcal{T} = (0,1)$, as follows:
\begin{equation}	\label{equ:model}
	\left\{\begin{aligned}
		&-\partial_{tt} u - \Delta_{g} u = f, && (t,x)\in \mathcal{T}\times \mathcal{M},\\
		&\partial_t u(0,x)=\mu_0(x), &&x\in \mathcal{M},\\
            &\partial_t u(1,x)=\mu_1(x), &&x\in \mathcal{M},
	\end{aligned}
	\right.
\end{equation}
Here, $\Delta_g$ denotes the Laplace–Beltrami operator on the manifold $\mathcal{M}$.

Such equations arise naturally in the context of the dynamic formulation of optimal transport problems on manifolds \cite{lavenant2018discreteot, dong2025ot, yu2023meanfieldmanifold, grange2023otsurface}. Based on the dynamic formulation introduced by Benamou and Brenier \cite{benamou2000cfdot}, the Wasserstein distance can be reformulated as a minimization problem over density functions defined on the time-space product manifold $\mathcal{T} \times \mathcal{M}$. This optimization problem is typically addressed using the alternating direction method of multipliers (ADMM) \cite{boyd2010admm}. A key computational step in ADMM involves solving the Laplace–Beltrami type equation \eqref{equ:model}. In particular, in our recent work \cite{dong2025ot}, we presented a superconvergent post-processing based justification for the empirical observation that ADMM remains convergent even after applying successive gradient and divergence operators to a linear finite element function. Furthermore, we proposed a gradient enhanced ADMM scheme that incorporates a gradient recovery procedure. The main objective of this paper is to establish the superconvergence theory for the gradient recovery operator on the product manifold and to provide a theoretical justification for the gradient enhanced ADMM algorithm developed in \cite{dong2025ot}.

The study of superconvergence theory dates back to the 1970s \cite{wahlbin1995superconvergencebook, chen1995high}, when it was observed that certain special points—referred to as superconvergent points—exhibit convergence rates higher than the optimal rates expected from polynomial interpolation. Post-processing techniques are among the most ubiquitous approaches for obtaining superconvergent results. Gradient recovery aims to construct a superconvergent discrete gradient by performing post-processing on the computed solution data. For problems defined on planar domains, this research area has reached a stage of maturity, with well-established theoretical foundations for superconvergence. Notable examples include the simple/weighted averaging method \cite{cc2002average}, the superconvergent patch recovery (SPR) \cite{zhu1992spr1, zhu1992spr2}, and the polynomial preserving recovery (PPR) \cite{naga2005ppr2d3d, zhang2005ppr, guo2017hessianrecovery}.
Recently, there has been growing interest in extending gradient recovery techniques to manifold settings \cite{du2005spr, wei2010sprsurface, dong2020pppr, dong2024super}. In \cite{dong2020pppr,dong2024super}, two of the authors introduced the parametric polynomial preserving recovery (PPPR) from an intrinsic geometric viewpoint and established its second-order consistency without requiring mesh symmetry or some a priori knowledge of the exact manifolds, including their tangential planes (or normal vectors).
For a comprehensive overview of  recovery techniques for post-processing purposes, interested readers are referred to the recent review paper \cite{guo2025recovery}.

For the Laplace–Beltrami type equation \eqref{equ:model} posed on the product manifold $\mathcal{T} \times \mathcal{M}$, it is not straightforward to apply standard finite element discretizations as in the planar case \cite{benamou2000cfdot}. To address this, we adopt a finite difference method for the temporal variable and a surface finite element method for the spatial variable, resulting in a coupled finite difference and surface finite element scheme (FDM-sFEM) \cite{dong2025ot}. We note that the topic of numerical solutions of partial differential equations on manifolds is itself a rapidly developing field; see, for example, \cite{dziuk1988sfem,dziuk2013review,dong2017,  bonito2019review, demlow2009highorder, guo2020scr,cai2024linear}. To obtain a superconvergent approximation of the gradient, we perform post-processing on the computed solution using the PPR in the temporal interval and the PPPR on the spatial manifold.

The primary contribution of this work is the development of a superconvergence theory on product manifolds. A key challenge in establishing such results lies in the absence of supercloseness estimates for the FDM–sFEM discretization. Moreover, the partial boundary condition is given by temporal nonhomogeneous Neumann conditions, which introduces additional complications. To address this, we derive a novel summation by parts formula that naturally integrates with the ghost penalty approach for handling Neumann boundary conditions. This new tool enables us to rigorously establish supercloseness results in the temporal direction. To overcome the gap in the spatial discretization, we construct an intermediate interpolation of the semi-discrete solution, for which the supercloseness is proved.

Another contribution of this work is the development of a novel geometric error analysis framework for sFEMs based on the Riemannian metric. Compared to existing geometric error analysis techniques for sFEMs, our approach considers differential operators on manifolds via their Riemannian metric tensors, avoiding the tangential projections, which are adopted by most works in the literature. There, the surface gradient is calculated by first extending the function in its ambient space, calculating the Euclidean gradient in ambient space, and then projecting into the tangent spaces.  
A key advantage of our approach is that it applies to more general geometric approximations and does not require exact information of the geometry. For instance, results in the literature often require the triangulated approximation to be an interpolation of the exact surface, and the normal vectors of the surface at the vertices are given in order to have the tangential projection, see, e.g., the seminal works \cite{dziuk1988sfem, dziuk2013review, demlow2009highorder}. 
In fact, the geometric error is more transparent to analyze using our approach: all geometric computations are systematically pulled back to local parametric domains.
Then our error analysis is built on approximations of Riemannian metric tensors (including their derivatives), and requires neither the exact vertices nor the normal vectors in the calculus.

The remainder of the paper is organized as follows. In Section~\ref{sec:pre}, we present a background on geometry and function spaces defined in the manifolds. Section~\ref{sec:met} is devoted to the numerical methods for the model equation, along with a brief review of gradient recovery techniques. In Section~\ref{sec:sup}, we first establish supercloseness results for the FDM-sFEM scheme and then demonstrate the superconvergence of the recovered gradient. Section~\ref{sec:num} presents numerical examples that support the theoretical findings. Finally, concluding remarks are given in Section~\ref{sec:con}.

\section{Preliminary}\label{sec:pre}
Throughout this work, the symbol $C$ (or $c$) denotes a generic positive constant that is independent of the mesh parameter $h$, whose value may differ at different occurrences. For notational convenience, we write $x \leq C y$ (or $x \geq C y$) as $x \lesssim y$ (or $x \gtrsim y$), respectively.

Our consideration begins with the time-space product manifold $\mathcal{T}\times \mathcal{M}$, where $\mathcal{T}=(0,1)$ is a time interval and $ \mathcal{M}\subset \mathbb{R}^n$ is a Riemannian manifold endowed with metric $g$.
The target is then developing numerical methods for discretizing the Laplace-Beltrami operator on this product manifold.
Due to the special structure of the product manifold, standard sFEMs from the literature \cite{dziuk1988sfem,demlow2009highorder} may not be optimal in efficiency. 
In this work, we combine the FDM in time and the sFEM in space for efficient numerical solutions.
While FDM is rather straight forward, the following introduction mostly focuses on the spatial manifold $\mathcal{M}$ and its approximations.

Let $\mathcal{M}_h = \bigcup\limits_{j \in J} \mathcal{M}_h^j$ be a triangulated approximation of the manifold $\mathcal{M}$, where $h$ denotes the size of the largest triangle and $J \subset \mathbb{N}$ denotes the index set. Its curved counterpart is denoted by $\mathcal{M} = \bigcup\limits_{j \in J}{ \mathcal{M}^j }$. From the definition of Riemannian manifolds, there exists a parametric domain $\Omega_h = \bigcup\limits_{j \in J} \Omega_h^j$ such that each patch $\mathcal{M}^j$ is locally diffeomorphic to $\Omega_h^j$. Let $\pi : \Omega_h \to \mathcal{M}$ denote this diffeomorphism. Then, the Riemannian metric $g$ can be computed via the pullback relation:
\begin{equation}\label{equ:metric}
g \circ \pi = (\partial \pi)^\top \partial \pi,
\end{equation}
where $\partial \pi$ denotes the Jacobian matrix of the mapping $\pi$.
Analogously, there exists a diffeomorphism $\pi_h : \Omega_h \rightarrow \mathcal{M}_h$, and the discrete Riemannian metric $g_h$ is given by
\begin{equation}\label{equ:discrete_metric}
g_h \circ \pi_h = (\partial \pi_h)^\top \partial \pi_h.
\end{equation}

It is worth noting that given $\Omega_h$ the mapping $\pi : \Omega_h \rightarrow \mathcal{M}$ is non-unique. However, the metric tensor $g$ is invariant with respect to the choice of $\pi$. In this work, we choose $\Omega_h$ to consist of the triangle faces (hyperplanes) of the triangular surface $\mathcal{M}_h$, and adopt the commonly used geometric mapping $\pi : \Omega_h \rightarrow \mathcal{M}$ defined by
\begin{equation}
\pi(x) = x - d(x)\nu(x), \quad \text{for all } x \in \Omega_h,
\label{equ:construction_of_pi(x)}
\end{equation}
where $\nu(x) = \nu(\pi(x))$ denotes the unit outward normal vector at the projected point $\pi(x) \in \mathcal{M}$, and $d(x)$ is the signed distance function to the manifold $\mathcal{M}$.

For the sake of simplicity, only linear approximation of the curved triangle $\mathcal{M}^j$ will be used, and higher-order approximation can be considered similarly as in \cite{dong2025dg}. Let $\left\{x_i\right\}_{i\in I}$ and $\left\{x_{i,h}\right\}_{i\in I}$  denote the set of vertices of the (curved) triangles in $\mathcal{M}$ and $\mathcal{M}_h$, respectively. To estimate the geometric error, we assume that
\begin{equation}\label{equ:mesh_assumption}
\max_{i\in I} | x_i - x_{i,h}| \lesssim h^2\; \text{ and } \;
\max_{i\in I} | P_{\Omega_h}x_i - P_{\Omega_h}x_{i,h}| \lesssim h^3,\;
\end{equation}
where $P_{\Omega_h}$ denotes the operation of projecting the vectors onto the common patch-wise parametric domain $\Omega_h$.

\begin{remark}
If $\mathcal{M}_h$ is obtained by the piecewise linear interpolation of the surface, as in the geometric error analysis in \cite{dziuk1988sfem, dziuk2013review, demlow2009highorder}, then we have $| x_i - x_{i,h}|=0$ for all $i \in I$, and the assumption \eqref{equ:mesh_assumption} is obviously satisfied.
\end{remark}

Based on assumption \eqref{equ:mesh_assumption}, the following approximation results for the Riemannian metric were established in \cite{dong2025dg}.

\begin{theorem}
Let \eqref{equ:mesh_assumption} be satisfied. Let $g$ be the  metric tensor of $\mathcal{M}$, and let $g_h$ denote the metric tensor of $\mathcal{M}_h$, which is a continuous piecewise linear approximation of $\mathcal{M}$. Then the following error bounds hold:
\begin{equation}\label{equ:metrictensor_approximation}
\begin{aligned}
&\left\|g^{-1}(g-g_h)\right\|_{L^{\infty}}\lesssim h^2,\quad \text{ and }\quad
&\left\|\frac{\sqrt{|g|}-\sqrt{|g_h|}}{\sqrt{|g|}}\right\|_{L^{\infty}}\lesssim h^2.
\end{aligned}
\end{equation}
\label{thm:metrictensor_approximation}
\end{theorem}

To obtain the superconvergent results, we should impose some constraints on the triangular mesh $\mathcal{M}_h$. Two adjacent triangles form an $\mathcal{O}(h^2)$ parallelogram if the lengths of their opposite edges differ by $\mathcal{O}(h^2)$. In the following, we assume that the triangular mesh $\mathcal{M}_h$ satisfies the following condition.

\begin{definition}
The triangulation mesh $\mathcal{M}_h$ is said to satisfy the $\mathcal{O}(h^{2\sigma})$ irregular condition if there exist a partition $\mathcal{M}_{1,h} \cup \mathcal{M}_{2,h}$ of $\mathcal{M}_h$ and a positive constant $\sigma$ such that every two adjacent triangles in $\mathcal{M}_{1,h}$ form an $\mathcal{O}(h^2)$ parallelogram, and  $ \sum\limits_{\mathcal{M}_h^j\in \mathcal{M}_{2,h}}|\mathcal{M}_h^j| = \mathcal{O}(h^{2\sigma})$.
\end{definition}

Let $\mathcal{V}(\mathcal{M})$ and $\mathcal{V}(\mathcal{M}_h)$ denote the ansatz function spaces on $\mathcal{M}$ and $\mathcal{M}_h$, where $\mathcal{V}$ may be taken as a Sobolev space such as $H^k$. We define a bijective transformation operator between $\mathcal{V}(\mathcal{M})$ and $\mathcal{V}(\mathcal{M}_h)$ as
\begin{equation}
\begin{aligned}
T_h: &\mathcal{V}(\mathcal{M})\rightarrow\mathcal{V}(\mathcal{M}_h),\\
&\omega\mapsto \omega \circ P_h,
\end{aligned}
\nonumber
\end{equation}
where $P_h=\pi\, \circ\, \pi_h^{-1}$ is a continuous and bijective projection map between each pair of corresponding elements $\mathcal{M}^j$ and $\mathcal{M}^j_h$.  Its inverse, denoted by $T_h^{-1}$,  is the bijective transformation operator from $\mathcal{V}(\mathcal{M}_h) $ to $\mathcal{V}(\mathcal{M})$.

For the transformation operator $T_h$, \cite{dziuk1988sfem,dong2020pppr} established the following result:

\begin{lemma}\label{lem:boundedness_of_T_h}
For a fixed $k \in \mathbb{N}$ and $p \geq 1$, let $\mathcal{V}(\mathcal{M}) \hookrightarrow W^{k,p}(\mathcal{M})$. 
If the smoothness of $\mathcal{M}_h$ is compatible with the regularity of $W^{k,p}$, then the operator $T_h$ is uniformly bounded between $W^{k,p}(\mathcal{M})$ and $W^{k,p}(\mathcal{M}_h)$ in the sense that
\begin{equation}
\begin{aligned}
           &\left\|T_h \omega\right\|_{W^{k,p}(\mathcal{M}_h)}\lesssim\|\omega\|_{W^{k,p}(\mathcal{M})}\lesssim\|T_h \omega\|_{W^{k,p}(\mathcal{M}_h)}, \quad \forall \omega\in\mathcal{V}(\mathcal{M}).
\end{aligned}
\label{equ:boundedness_of_T_h}
\end{equation}
\end{lemma}


Using the Riemannian metric, the tangential gradient can be defined as
	\begin{equation}\label{equ:surfacegradient}
		\nabla_g \omega=g^{ij}\partial_i\omega\partial_j,
	\end{equation}
and the Laplace–Beltrami operator is given by
	\begin{equation}\label{equ:laplacebeltrami}
\Delta_g \omega = \frac{1}{\sqrt{|g|}} \partial_i \left( \sqrt{|g|} g^{ij} \partial_j \omega \right),
\end{equation}
where $\partial_j$ denotes the tangential basis,  $g^{ij}$ is the $(i,j)$-entry of the inverse Riemannian metric tensor $g$, and  $|g| := |\det(g_{ij})|$ is the absolute value of the determinant of the metric tensor.

Under the parametrization map $\pi$, the tangential gradient $\nabla_g$ can be realized as
\begin{equation}\label{equ:surfacegradient_realization}
	(\nabla_g \omega)\circ \pi = \partial \pi (g^{-1}\circ \pi) \nabla \bar{\omega},
\end{equation}
where $\bar{\omega} = \omega \circ \pi$ is the pullback of the function $u$ to the local planar parametric domain $\Omega_h$.
Similarly, we define the discrete tangential operator $\nabla_{g_h}$. We emphasize that all differential operators are defined piecewise. 

The Riemann metric tensor provides a different way of computing the $L^2$ product of gradients on $\mathcal{M}$ and $\mathcal{M}_h$. For any $\omega,\psi\in H^1(\mathcal{M})$, and $\omega_h, \psi_h\in H^1(\mathcal{M}_h)$, the following properties hold \cite{dong2025dg}:
\begin{equation}
         \begin{aligned}
             &\int_{\mathcal{M}}\nabla_{g}\omega\cdot\nabla_{g}\psi\mathrm{d}\sigma_{g}=\sum_{j\in J}\int_{\Omega_h^j}(\nabla\bar{\omega})^{\top}g^{-1}\nabla\bar{\psi}\sqrt{|g|}\mathrm{d}\sigma,\\
             &\int_{\mathcal{M}_h}\nabla_{g_h}\omega_h\cdot\nabla_{g_h}\psi_h\mathrm{d}\sigma_{g_h}=\sum_{j\in J}\int_{\Omega_h^j}(\nabla\bar{\omega}_h)^{\top}g_h^{-1}\nabla\bar{\psi}_h\sqrt{|g_h|}\mathrm{d}\sigma,
         \end{aligned}
         \label{eq:Theorem:Integral_of_gradients_transform}
     \end{equation}
where $\mathrm{d}\sigma,\mathrm{d}\sigma_{g}$, and $\mathrm{d}\sigma_{g_h}$ represent measures on the parametric domain, the exact surface and the approximate surface, respectively, and $\mathrm{d}\sigma_g=\sqrt{|g|}\,\mathrm{d}\sigma$, $\mathrm{d}\sigma_{g_h}=\sqrt{|g_h|}\,\mathrm{d}\sigma$.


\section{Numerical methods }\label{sec:met}
In this section, we consider the numerical approximation of the Laplace–Beltrami type equation \eqref{equ:model}. The first part is devoted to the construction of a fully discrete scheme based on a hybrid approach combining finite difference and finite element methods. In the subsequent subsection, we introduce gradient recovery techniques to  produce a superconvergent post-processed gradient.

\subsection{FDM-sFEM Scheme}\label{ssec:sch}
Motivated by applications in optimal transport \cite{benamou2000cfdot,dong2025ot}, the model problem \eqref{equ:model} incorporates a Neumann type boundary condition in time. The solution to \eqref{equ:model} is determined only up to an additive constant. To guarantee well-posedness, we impose the following normalization condition:
\begin{equation}\label{equ:sol_constrain}
\int_{\mathcal{T}}\int_{\mathcal{M}} u \, \mathrm{d}\sigma_g\, \mathrm{d}t = 0.
\end{equation}

In addition, the data must satisfy the following compatibility condition:
\begin{equation}
\int_{\mathcal{T}}\int_{\mathcal{M}} f(x,t)\,  \mathrm{d}\sigma_g \, \mathrm{d}t = \int_{\mathcal{M}} \left( \mu_0(x) - \mu_1(x) \right) \, \mathrm{d}\sigma_g.
\end{equation}

We begin by discretizing the temporal domain in \eqref{equ:model} using the central finite difference method. To this end, we partition the time interval $\mathcal{T}$ into $N$ uniform subintervals:
\begin{equation}	\label{equ:time_intervals}
	0 = t_0 < t_1 < t_2 < \cdots < t_{N-1} < t_{N} = 1,
\end{equation}
with step size $\tau = \frac{1}{N}$ and $t_i = i\tau $ for $i=0,\ldots,N$.  We define the first-order difference operator $D_t$ as
\begin{equation}
    D_{t}u^i(x)=\frac{u^i(x)-u^{i-1}(x)}{\tau},\quad i=0,\ldots,N,
    \label{equ:first_difference}
\end{equation}
and the second-order difference operator $D_{tt}$ is defined as
 \begin{equation} \label{equ:second_difference}
    D_{tt}u^i=\frac{1}{\tau}(D_{t}u^{i+1}-D_{t}u^{i}),\quad i=0,\ldots, N.
\end{equation}

To incorporate the Neumann type boundary condition in time, we adopt a ghost penalty inspired modification for the second-order finite difference operator $D_{tt}$. Specifically, we define
\begin{equation}\label{equ:difference_operator}
    D_{tt}u^i(x)=\left\{\begin{aligned}
        &\frac{2u^1(x)-2u^0(x)}{\tau^2}, &&i=0,\\
        &\frac{u^{i+1}(x)-2u^{i}(x)+u^{i-1}(x)}{\tau^2}, &&1\leq i\leq N-1,\\
        &\frac{2u^{N-1}(x)-2u^{N}(x)}{\tau^2}, &&i=N.
    \end{aligned}\right.
\end{equation}

Additionally, we  define the auxiliary  right hand side function  as
\begin{equation} \label{equ:boundary_correction}
	b^i = 
	\left \{
	\begin{array}{ll}
		\frac{2}{\tau}\mu_0,   & i =  0, \\[4pt]
		0, & 1 \le i \le N-1, \\[4pt]
		-\frac{2}{\tau}\mu_1, & i = N.
	\end{array}
	\right.
\end{equation}

The semi-discrete formulation of the model equation \eqref{equ:model} reads as: find $u^i \in H^1(\mathcal{M})$ such that
\begin{equation}\label{equ:semi_discretization}
	- D_{tt} u^i - \Delta_g u^i = f(t_i) + b^i,
\end{equation}
with the corresponding weak form given by
\begin{equation}\label{equ:semi_discretization_weak}
    - (D_{tt} u^i, v)_{\mathcal{M}} + (\nabla_g u^i, \nabla_g v)_{\mathcal{M}} = (f(t_i) + b^i, v)_{\mathcal{M}}, \quad \forall v \in H^1(\mathcal{M}),
\end{equation}
for $i = 0, \ldots, N$. In \eqref{equ:semi_discretization_weak}, $(\cdot,\cdot)_{\mathcal{M}}$ denotes the $L^2$-inner product on $\mathcal{M}$.

We define the following piecewise linear basis functions in the temporal direction:
\begin{equation}
	\phi^i(t) = \begin{cases}
		\displaystyle \frac{t - t_{i-1}}{\tau}, & t \in [t_{i-1}, t_i], \\
		\displaystyle \frac{t_{i+1} - t}{\tau}, & t \in [t_i, t_{i+1}], \\
		0, & \text{otherwise},
	\end{cases}
	\label{temporal_linear_basis_function}
\end{equation}
for $i=1,\ldots, N-1$, with the boundary terms $\phi_0(t)=(t_1-t)/\tau$ for $t\in[t_{0},t_1]$ and $\phi_0(t)=0$ for $t\notin [t_{0},t_1]$, $\phi_N(t)=(t-t_{N-1})/\tau$ for $t\in[t_{N-1},t_N]$ and $\phi_N(t)=0$ for $t\notin [t_{N-1},t_N]$.

Then, the semi-discrete FDM solution can be written as
\begin{equation}\label{equ:semi-discrete_solution}
	u_N(t, x) = \sum_{i = 0}^{N} u^i(x) \phi^i(t).
\end{equation}

To obtain a fully discrete approximation, we adopt a hybrid approach by considering a continuous piecewise linear surface finite element space on the discrete surface $\mathcal{M}_h$, defined as
\begin{equation}\label{equ:fem_space}
    S_h = \left\{ v_h  \in H^1(\mathcal{M}_h)\ \middle|\ \bar{v}_h=v_h\circ \pi_h \in \mathbb{P}^1(\Omega_h^j)\ \text{for all } j \in J \right\},
\end{equation}
where $\mathbb{P}^1(\Omega_h^j)$ denotes the space of piecewise linear polynomials on every parametric patch $\Omega_h^j$. The corresponding lifted space on $\mathcal{M}$ is denoted by
\begin{equation}\label{equ:fem_space_exact_surface}
    \tilde{S}_h = \left\{ v \in H^1(\mathcal{M})  \ \middle|\ v = T_h^{-1} v_h \text{ and } v_h  \in S_h\ \right\}.
\end{equation}

The FDM-sFEM discretization of the model equation \eqref{equ:model} then seeks $u_h^i \in S_h$ such that
\begin{equation}\label{equ:full_discrete}
	- (D_{tt} u_h^i, v_h)_{\mathcal{M}_h} + (\nabla_{g_h} u_h^i, \nabla_{g_h} v_h)_{\mathcal{M}_h} = (T_h \left(f(t_i) + b^i\right), v_h)_{\mathcal{M}_h}, \quad \forall v_h \in S_h,
\end{equation}
for $i = 0, \ldots, N$.

The fully discrete FDM-sFEM solution is represented as
\begin{equation}\label{equ:full_discrete_solution}
	u_h(t, x) = \sum_{i = 0}^{N} u_h^i(x) \phi^i(t).
\end{equation}

\begin{remark}
	To solve the FDM-sFEM scheme in \eqref{equ:full_discrete}, one must address a large-scale linear system. A fast solver based on the reduction dimension using the eigenvalue decomposition in the time direction has been developed in \cite{dong2025ot}.
\end{remark}

\subsection{Gradient recovery techniques}\label{ssec:grt}
In this subsection, we investigate the superconvergent post-processing of the FDM-sFEM solutions. Specifically, we employ the polynomial preserving recovery (PPR) technique to enhance the accuracy of the temporal derivative, while the parametric polynomial preserving recovery (PPPR) method is utilized for recovering the surface gradient.

\subsubsection{Polynomial preserving recovery}\label{sssec:ppr}
Let $S_{\tau}=\operatorname{span}\{\phi_i(t)\}$ denote the space of continuous piecewise linear functions defined on a uniform partition of the time interval  ${\mathcal{T}}$. The  PPR gradient recovery operator  $G_{\tau}: S_{\tau}\rightarrow S_{\tau}$,  following \cite{zhang2005ppr, guo2025recovery}, is constructed in three steps.

First, for each node $t_i$, we define the corresponding local patch: 
 \begin{equation}
      I_{t_i}=\left\{\begin{aligned}
          &(t_0,t_2), &&i=0,\\
          &(t_i,t_{i+1}), && 1\leq i\leq N-1,\\
          &(t_{N-2},t_{N}), &&i=N.
      \end{aligned}\right.
  \end{equation}
  
Next, over each patch $I_{t_i}$, we construct a quadratic polynomial $p_{t_i}$ approximating $\omega_{\tau}\in S_{\tau}$ by solving the local least-squares problem:
  \begin{equation}
      p_{t_i}=\arg\min\limits_{p\in\mathbb{P}^2(I_{t_i})}\sum_{t_k\in I_{t_i}\cap\mathcal{N}_{\tau}}\left|p(t_k)-\omega_{\tau}(t_k)\right|^2,
  \end{equation}
  where $\mathcal{N}_{\tau}=\{t_0,t_1,\ldots,t_{N}\}$ denotes the set of nodal points. The recovered gradient at $t_i$ is then defined by evaluating the derivative of the fitted polynomial:
    \begin{equation}
      (G_{\tau} \omega_{\tau})(t_i)=\left.\frac{\mathrm{d}p_{t_i}(t)}{\mathrm{d}t}\right|_{t=t_i}.
  \end{equation}
  
Finally, the global recovered function $G_{\tau} \omega_{\tau} \in S_{\tau}$ is constructed via nodal interpolation:
  \begin{equation}
     G_{\tau}\omega_{\tau} =\sum_{i=0}^{N}(G_{\tau}\omega_{\tau})(t_i)\phi_i(t).
  \end{equation}

The operator $G_{\tau}$ satisfies the following consistency result:
\begin{lemma}	\label{lem:PPR_consistency}
        Let $G_{\tau}$ be the PPR gradient recovery operator.  For any $\omega\in H^{3}({\mathcal{T}})$,  it holds that
    \begin{equation}
    		\left\|G_{\tau} I_{\tau} \omega-\partial_t \omega\right\|_{L^2(T)}\lesssim \tau^{2}\left\|\omega\right\|_{H^{3}(T)},
                \label{eq:lemma:G_tau_closeness}
    \end{equation}
    where $I_{\tau}$ is the Lagrange interpolation operator on $\tau$. 
\end{lemma}
  
In addition, $G_{\tau}$ is bounded in the following sense:
    \begin{lemma}        \label{lem:G_tau_boundedness}
        Let $G_{\tau}$ be the PPR gradient recovery operator.  For any $\omega\in H^{1}({\mathcal{T}})$,  it holds that
        \begin{equation}
            \left\|G_{\tau}\omega\right\|_{L^2(T)}\lesssim\left\|\partial_t \omega\right\|_{L^2(T)}.
        \end{equation}
    \end{lemma}

\subsubsection{Parametric polynomial preserving recovery}\label{sssec:pppr}
To recover the surface gradient, \cite{dong2020pppr} proposes an intrinsic approach based on the geometric realization of the tangential gradient. Combining equations \eqref{equ:metric} and \eqref{equ:surfacegradient_realization}, we obtain
\begin{equation}
	(\nabla_g \omega) \circ \pi = \left( (\partial \pi)^{\dagger} \right)^{\top} \nabla \bar{\omega},
\end{equation}
where $(\partial \pi)^{\dagger}$ denotes the Moore--Penrose pseudo inverse of the Jacobian matrix $\partial \pi$. This relation indicates that to construct a superconvergent recovered tangential gradient, it suffices to obtain superconvergent approximations of both $\partial \pi$ and $\nabla \bar{\omega}$ on the parametric domain $\Omega_h$.

Let $\bar{G}_h$ denote the PPR gradient recovery operator defined on $\Omega_h$. Then, the parametric polynomial preserving recovery (PPPR) operator $G_h: S_h \rightarrow S_h$, as introduced in \cite{dong2020pppr}, is given by
\begin{equation}
	(G_h \omega_h) \circ \pi_h = \left(\left( \bar{G}_h \pi_h \right)^{\dagger}\right)^{\top} \bar{G}_h \bar{\omega},
\end{equation}
where $\bar{\omega}$ denotes the lifted function associated with $v_h$, and $\pi_h$ is the parametrization map.

Similar to the temporal PPR operator $G_{\tau}$, the operator $G_h$ satisfies a second-order consistency estimate \cite{dong2020pppr}:

\begin{lemma}\label{lem:pppr_consistency}
Suppose $\omega \in W^{3}_{\infty}(\mathcal{M})$. Then the following estimate holds:
\begin{equation}\label{equ:pppr_consistency}
	\left\| \nabla_g \omega- T_h^{-1} G_h \left( I_h T_h \omega \right) \right\|_{L^2(\mathcal{M})}
	\leq h^2 \sqrt{\mathcal{A}(\mathcal{M})} D(g, g^{-1}) \left\| \omega \right\|_{W^{3}_{\infty}(\mathcal{M})},
\end{equation}
where $I_h$ denotes the spatial linear interpolation operator on $\mathcal{M}_h$, $D(g, g^{-1})$ is a constant depending on the metric tensor $g$ and its inverse, and $\mathcal{A}(\mathcal{M})$ denotes the surface area of $\mathcal{M}$.
\end{lemma}

In addition, the PPPR operator $G_h$ satisfies the following boundedness property \cite{dong2020pppr}:
\begin{lemma}\label{lem:pppr_bounded}
For any $\omega_h \in \mathcal{V}(\mathcal{M}_h)$, the operator $G_h$ is bounded in the sense that
\begin{equation}\label{equ:pppr_bounded}
	\left\| G_h \omega_h \right\|_{L^2(\mathcal{M}_h)} \lesssim \left\| \nabla_{g_h} \omega_h \right\|_{L^2(\mathcal{M}_h)},
\end{equation}
where the hidden constant depends on the geometry of $\mathcal{M}$ and the shape regularity of the triangulation $\mathcal{M}_h$, but is independent of the mesh size $h$.
\end{lemma}

\section{Superconvergent analysis}\label{sec:sup}
In this section, we establish superconvergence results for the gradient recovery operator applied to the FDM-sFEM solution. In subsection~\ref{ssec:tmp}, we derive a supercloseness estimate between the semi-discrete solution and its piecewise linear interpolant in the temporal direction. Subsection~\ref{ssec:spa} focuses on the spatial supercloseness. These results are then combined in  subsection~\ref{ssec:srg} to prove the superconvergence of the recovered gradient.

\subsection{Temporal supercloseness}\label{ssec:tmp}
For any $\omega \in H^1(\mathcal{T}; L^2(\mathcal{M}))$, we define the temporal interpolation operator $I_{\tau}$ by
\begin{equation}\label{equ:temporal_interpolation}
	I_{\tau} \omega(t) = \sum_{i=0}^{N} \omega(t_i) \, \phi_i(t),
\end{equation}
where $\{ \phi_i(t) \}$ are the piecewise linear basis functions in time as defined in \eqref{temporal_linear_basis_function}. 

The interpolation operator $I_{\tau}$ satisfies the following approximation properties; see \cite{he2018fdmfem} for details.
\begin{lemma}\label{lem:closeness_temporal_interpolation}
Let $\omega \in H^1(\mathcal{T}; L^2(\mathcal{M}))$. Then the following estimates hold:
\begin{equation}
	\begin{aligned}
		\| \omega - I_{\tau} \omega \|_{L^2(\mathcal{T}; L^2(\mathcal{M}))} &\leq \tau \, \| \partial_t (\omega - I_{\tau} \omega) \|_{L^2(\mathcal{T}; L^2(\mathcal{M}))}, \\
		\| \partial_t I_{\tau} \omega \|_{L^2(\mathcal{T}; L^2(\mathcal{M}))} &\leq \| \partial_t \omega \|_{L^2(\mathcal{T}; L^2(\mathcal{M}))}.
	\end{aligned}
\end{equation}

Moreover, if $\omega \in H^2(\mathcal{T}; L^2(\mathcal{M}))$, then
\begin{equation}
	\| \partial_t (\omega - I_{\tau} \omega) \|_{L^2(\mathcal{T}; L^2(\mathcal{M}))} \leq \tau \, \| \partial_{tt} \omega \|_{L^2(\mathcal{T}; L^2(\mathcal{M}))}.
\end{equation}
\end{lemma}

One of the key ingredients for establishing supercloseness in time is the use of summation by parts (also referred to as the discrete Green’s formula), which serves as a discrete analogue of integration by parts. For sequences $\{\alpha^i\},\ \{\beta^i\}\subset L^2(\mathcal{M})$, the classical form of summation by parts can be written as
\begin{equation}    \label{equ:classical_summation_by_parts}
       \sum_{i=1}^{N-1}(\alpha^{i+1}-\alpha^i,\beta^i)_{\mathcal{M}}+(\alpha^1,\beta^0)_{\mathcal{M}}-(\alpha^{N},\beta^{N})_{\mathcal{M}}=-\sum_{i=1}^{N}(\alpha^i,\beta^{i}-\beta^{i-1})_{\mathcal{M}}.
\end{equation}

To incorporate with the ghost penalty method for the Neumann boundary condition, we require the following summation by parts formula:
\begin{lemma} Let $D_{tt}$ denote the finite difference operator defined in \eqref{equ:difference_operator}. For sequences $\{\alpha^i\},\ \{\beta^i\}\subset L^2(\mathcal{M})$, the following identity holds:
\begin{equation}\label{equ:summation_by_parts}
	\sum_{i=0}^{N}w_i(D_{tt}\alpha^i,\beta^i)_{\mathcal{M}}=-\sum_{i=1}^{N}\left(D_t \alpha^i,D_{t}\beta^i\right)_{\mathcal{M}},
\end{equation}
where the weights are defined by $w_0 = w_N = 1/2$ and $w_i = 1$ for $i = 1, \ldots, N-1$.
\end{lemma}
\begin{proof}
	By the definition of $w_i$ and the finite difference operator, we have 
	\begin{equation*}
    \begin{aligned}
        &\sum_{i=0}^{N}w_i(D_{tt}\alpha^i,\beta^i)_{\mathcal{M}}={\frac{1}{2}}\left(D_{tt}\alpha^0,\beta^0\right)_{\mathcal{M}}+{\frac{1}{2}}\left(D_{tt}\alpha^{N},\beta^{N}\right)+\sum_{i=1}^{N-1}(D_{tt}\alpha^i,\beta^i)_{\mathcal{M}}\\
        =&\left(\frac{\alpha^1-\alpha^0}{\tau^2},\beta^0\right)_{\mathcal{M}}+\left(\frac{\alpha^{N-1}-\alpha^{N}}{\tau^2},\beta^{N}\right)_{\mathcal{M}}+\sum_{i=1}^{N-1}\frac{1}{\tau}\left[\left(\frac{\alpha^{i+1}-\alpha^{i}}{\tau},\beta^i\right)_{\mathcal{M}}-\left(\frac{\alpha^{i}-\alpha^{i-1}}{\tau},\beta^i\right)_{\mathcal{M}}\right]\\
        =&\left(\frac{\alpha^1-\alpha^0}{\tau^2},\beta^0\right)_{\mathcal{M}}+\left(\frac{\alpha^{N-1}-\alpha^{N}}{\tau^2},\beta^{N}\right)_{\mathcal{M}}+\sum_{i=2}^{N}\frac{1}{\tau}\left(\frac{\alpha^{i}-\alpha^{i-1}}{\tau},\beta^{i-1}\right)_{\mathcal{M}}-\sum_{i=1}^{N-1}\frac{1}{\tau}\left(\frac{\alpha^{i}-\alpha^{i-1}}{\tau},\beta^i\right)_{\mathcal{M}}\\
        =&\left(\frac{\alpha^1-\alpha^0}{\tau^2},\beta^0\right)_{\mathcal{M}}+\left(\frac{\alpha^{N-1}-\alpha^{N}}{\tau^2},\beta^{N}\right)_{\mathcal{M}}-\sum_{i=2}^{N-1}\left(\frac{\alpha^{i}-\alpha^{i-1}}{\tau},\frac{\beta^i-\beta^{i-1}}{\tau}\right)_{\mathcal{M}}\\
        &+\frac{1}{\tau}\left(\frac{\alpha^{N}-\alpha^{N-1}}{\tau},\beta^{N-1}\right)_{\mathcal{M}}-\frac{1}{\tau}\left(\frac{\alpha^1-\alpha^0}{\tau},\beta^1\right)_{\mathcal{M}}\\
        =&-\left(\frac{\alpha^1-\alpha^0}{\tau},\frac{\beta^1-\beta^0}{\tau}\right)_{\mathcal{M}}-\left(\frac{\alpha^{N}-\alpha^{N-1}}{\tau},\frac{\beta^{N}-\beta^{N-1}}{\tau}\right)_{\mathcal{M}}-\sum_{i=2}^{N-1}\left(D_t \alpha^i,D_{t}\beta^i\right)_{\mathcal{M}}\\
        =&-\sum_{i=1}^{N}\left(D_t \alpha^i,D_{t}\beta^i\right)_{\mathcal{M}},
    \end{aligned}
\end{equation*}
which completes the proof.
\end{proof}

For the finite difference scheme defined in \eqref{equ:difference_operator}, we establish the following truncation error estimates:
\begin{lemma}\label{lem:truncation_error}
    For $u\in H^3(\mathcal{T};L^2(\mathcal{M}))$, let $D_{tt}$ be defined as \eqref{equ:difference_operator},  then we have
    \begin{equation}\label{equ:truncation_error}
		\partial_{tt}u(t_i)-D_{tt}u(t_i)=\left\{\begin{aligned}&-\frac{1}{\tau}G^1-\frac{2}{\tau}\mu_0, &&i=0;\\
		    &-\frac{1}{2\tau}\left(G^{i+1}-G^i\right)+\frac{1}{\tau}\int_{t_{i-1}}^{t_i}\left(t-\frac{t_{i-1}+t_i}{2}\right)\partial_{ttt}u(t)\,\mathrm{d}t, \quad &&1\leq i\leq N-1;\\
            &\frac{1}{\tau}G^{N}+\frac{2}{\tau}\int_{t_{N-1}}^{t_{N}}\left(t-\frac{t_{N-1}+t_{N}}{2}\right)\partial_{ttt}u(t)\,\mathrm{d}t+\frac{2}{\tau}\mu_1, &&i=N,
		\end{aligned}\right.
	\end{equation}
	where $G^i=\frac{1}{\tau}\int_{t_{i-1}}^{t_i}(t-t_i)^2\partial_{ttt}u(t)\,\mathrm{d}t$ for $i=1,\ldots, N$.
\end{lemma}
\begin{proof}
    For $i = 1, \ldots, N - 1$, the formula has already been established in \cite{he2018fdmfem}.  
    
    For $i = 0$, using integration by parts, we obtain
    \begin{equation*}
        \begin{aligned}
            \partial_{tt}u(t_0)-D_{tt}u(t_0)= 
            &\partial_{tt}u(t_0)-\frac{2u(t_1)-2u(t_0)}{\tau^2}
            =\partial_{tt}u(t_0)-\frac{2}{\tau^2}\int_{t_0}^{t_1}\partial_tu(t)\,\mathrm{d}t\\
            =&\partial_{tt}u(t_0)-\frac{2}{\tau^2}\left(\tau\partial_tu(t_0)-\int_{t_0}^{t_1}(t-t_1)\partial_{tt}u(t)\,\mathrm{d}t\right)\\
            =&\partial_{tt}u(t_0)+\frac{2}{\tau^2}\int_{t_0}^{t_1}(t-t_1)\partial_{tt}u(t)\,\mathrm{d}t-\frac{2}{\tau}\mu_0\\
            =&\partial_{tt}u(t_0)+\frac{2}{\tau^2}\left(-\frac{\tau^2}{2}\partial_{tt}u(t_0)-\frac{1}{2}\int_{t_0}^{t_1}(t-t_1)^2\partial_{ttt}u(t)\,\mathrm{d}t\right)-\frac{2}{\tau}\mu_0\\
            =&-\frac{1}{\tau^2}\int_{t_0}^{t_1}(t-t_1)^2\partial_{ttt}u(t)\,\mathrm{d}t-\frac{2}{\tau}\mu_0=-\frac{1}{\tau}G^1-\frac{2}{\tau}\mu_0.
        \end{aligned}
    \end{equation*}
    
Using the same argument, we can deduce the case when $i=N$:
   \begin{equation*}
       \begin{aligned}
           &\partial_{tt}u(t_{N})-D_{tt}u(t_{N})=\partial_{tt}u(t_{N})-\frac{2u(t_{N-1})-2u(t_{N})}{\tau^2}\\
           =&\frac{1}{\tau^2}\int_{t_{N-1}}^{t_{N}}(t-t_{N-1})^2\partial_{ttt}u(t)\mathrm{d}t+\frac{2}{\tau}\mu_1\\
           =&\frac{1}{\tau^2}\int_{t_{N-1}}^{t_{N}}(t-t_{N})^2\partial_{ttt}u(t)\mathrm{d}t+\frac{1}{\tau^2}\int_{t_{N-1}}^{t_{N}}\left(t-t_{N-1}+t-t_{N}\right)\left(t-t_{N-1}-t+t_{N}\right)\mathrm{d}t+\frac{2}{\tau}\mu_1\\
           =&\frac{1}{\tau}G^{N}+\frac{2}{\tau}\int_{t_{N-1}}^{t_{N}}\left(t-\frac{t_{N-1}+t_{N}}{2}\right)\partial_{ttt}u(t)\mathrm{d}t+\frac{2}{\tau}\mu_1,
       \end{aligned}
   \end{equation*}
  which completes our proof.   
\end{proof}

Now, we are in the position to present the main results of  supercloseness results in temporal direction. 
\begin{lemma}\label{lem:temporal}
Let  $u\in H^4(\mathcal{T};L^2(\mathcal{M}))\cap H^2(\mathcal{T};H^2(\mathcal{M}))$ be the solution of \eqref{equ:model}, $u_N$ be the semi-discrete FDM solution defined in \eqref{equ:semi-discrete_solution}, and $I_{\tau}$ be the temporal interpolation operator defined in \eqref{equ:temporal_interpolation}. Then the following supercloseness estimate holds:
\begin{equation}
\begin{aligned}
&\| \nabla_g (I_{\tau} u - u_N) \|^2_{L^2(\mathcal{T}; L^2(\mathcal{M}))} + \| \partial_{t} (I_{\tau} u - u_N) \|^2_{L^2(\mathcal{T}; L^2(\mathcal{M}))} \\
\lesssim \,  &\tau^4 \big( 
 \| \nabla_g \partial_{tt} u \|^2_{L^2(\mathcal{T}; L^2(\mathcal{M}))}
+ \| \nabla_g \Delta_g^{-1} \partial_{tt} f \|^2_{L^2(\mathcal{T}; L^2(\mathcal{M}))} 
 + \| \partial_{ttt} u \|^2_{L^2(\mathcal{T}; L^2(\mathcal{M}))}
\big).
\end{aligned}
\end{equation}
\end{lemma}

\begin{proof}
	Let $u(t_i)$ denote the exact solution at $t_i$ and set $e^i = u(t_i) - u^i$.  Using the semi-discretization \eqref{equ:semi_discretization}, we can deduce that 
	\begin{equation}
				-D_{tt}e^i-\Delta_g e^i=\partial_{tt}u(t_i)-D_{tt}u(t_i) - b^{i},\quad \text{for }i=0,\ldots, N.
	\end{equation}
    
	First, we consider the interior points. For $i=1,\ldots, N-1$,  Lemma~\ref{lem:truncation_error}  implies that 
	\begin{equation}\label{equ:error_equation}
		-D_{tt}e^i-\Delta_g e^i=-\frac{1}{2\tau}\left(G^{i+1}-G^i\right)+\frac{1}{\tau}\int_{t_{i-1}}^{t_i}\left(t-\frac{t_{i-1}+t_i}{2}\right)\partial_{ttt}u(t)\,\mathrm{d}t.
	\end{equation}
    
	Using integration by parts, we obtain that 
		\begin{equation}		\label{equ:relations_between_E_and_G_rhs}
		\begin{aligned}
			&2\int_{t_{i-1}}^{t_i}\left(t-\frac{t_{i-1}+t_i}{2}\right)\partial_{ttt}u(t)\,\mathrm{d}t
			=\int_{t_{i-1}}^{t_i}\left(\left(\frac{\tau}{2}\right)^2-\left(t-\frac{t_{i-1}+t_i}{2}\right)^2\right)\partial_{tttt}u(t)\,\mathrm{d}t\\
			=&\int_{t_{i-1}}^{t_i}\left(\left(\frac{\tau}{2}\right)^2-\left(t-\frac{t_{i-1}+t_i}{2}\right)^2\right)\partial_{tt}\left(-\Delta_g u-f\right)\,\mathrm{d}t\\
			=&\int_{t_{i-1}}^{t_i}\left(\left(\frac{\tau}{2}\right)^2-\left(t-\frac{t_{i-1}+t_i}{2}\right)^2\right)\Delta_g \left(-\partial_{tt}u-\Delta_g^{-1}\partial_{tt}f\right)\,\mathrm{d}t.\\
		\end{aligned}
	\end{equation}
    
        Inserting  \eqref{equ:relations_between_E_and_G_rhs} into \eqref{equ:error_equation}, taking the inner product of the resulting expression with $2\tau e^i$ over $\mathcal{M}$, we obtain	
        \begin{equation}
		\begin{aligned}
			&-2\tau(D_{tt}e^i,e^i)_{\mathcal{M}}+2\tau\left\|\nabla_g e^i\right\|_{L^2(\mathcal{M})}^2+(G^{i+1}-G^i,e^i)_{\mathcal{M}}\\
			=&\left(\int_{t_{i-1}}^{t_i}\left(\left(\frac{\tau}{2}\right)^2-\left(t-\frac{t_{i-1}+t_i}{2}\right)^2\right)\Delta_g \left(-\partial_{tt}u-\Delta_g^{-1}\partial_{tt}f\right)\,\mathrm{d}t,e^i\right)_{\mathcal{M}}\\
			=&\left(\int_{t_{i-1}}^{t_i}\left(\left(\frac{\tau}{2}\right)^2-\left(t-\frac{t_{i-1}+t_i}{2}\right)^2\right) \left(\nabla_g\partial_{tt}u+\nabla_g\Delta_g^{-1}\partial_{tt}f\right)\,\mathrm{d}t,\nabla_g e^i\right)_{\mathcal{M}}\\
            \leq& \tau\left\|\nabla_g e^i\right\|_{L^2(\mathcal{M})}^2 +\frac{1}{\tau}\left\|\int_{t_{i-1}}^{t_i}\left(\left(\frac{\tau}{2}\right)^2-\left(t-\frac{t_{i-1}+t_i}{2}\right)^2\right) \left(\nabla_g\partial_{tt}u+\nabla_g\Delta_g^{-1}\partial_{tt}f\right)\,\mathrm{d}t\right\|_{L^2(\mathcal{M})}^2\\
            \leq& \tau\left\|\nabla_g e^i\right\|_{L^2(\mathcal{M})}^2 +\frac{1}{\tau}\int_{t_{i-1}}^{t_i}\left(\left(\frac{\tau}{2}\right)^2-\left(t-\frac{t_{i-1}+t_i}{2}\right)^2\right)^2\,\mathrm{d}t\int_{t_{i-1}}^{t_i}\left\|\nabla_g\partial_{tt}u+\nabla_g\Delta_g^{-1}\partial_{tt}f\right\|_{L^2(\mathcal{M})}^2\,\mathrm{d}t\\
            \leq&\tau\left\|\nabla_g e^i\right\|_{L^2(\mathcal{M})}^2 +\frac{1}{\tau}\int_{t_{i-1}}^{t_i}\left(\left(\frac{\tau}{2}\right)^2\right)^2\,\mathrm{d}t\int_{t_{i-1}}^{t_i}\left\|\nabla_g\partial_{tt}u+\nabla_g\Delta_g^{-1}\partial_{tt}f\right\|_{L^2(\mathcal{M})}^2\,\mathrm{d}t\\
                =&\tau\left\|\nabla_g e^i\right\|_{L^2(\mathcal{M})}^2+\left(\frac{\tau}{2}\right)^4\int_{t_{i-1}}^{t_i}\left\|\nabla_g\partial_{tt}u+\nabla_g\Delta_g^{-1}\partial_{tt}f\right\|_{L^2(\mathcal{M})}^2\,\mathrm{d}t\\
                \leq&\tau\left\|\nabla_g e^i\right\|_{L^2(\mathcal{M})}^2+\tau^4\int_{t_{i-1}}^{t_i}\left(\left\|\nabla_g\partial_{tt}u\right\|_{L^2(\mathcal{M})}^2+\left\|\nabla_g\Delta_g^{-1}\partial_{tt}f\right\|_{L^2(\mathcal{M})}^2\right)\,\mathrm{d}t,
		\end{aligned}
		\label{relations_between_E_and_G_3}
	\end{equation}
        where the first inequality follows from Young’s inequality, the second from H\"older’s inequality, and the final estimate uses the fact that
        \begin{equation}
            0\leq\left(\frac{\tau}{2}\right)^2-\left(t-\frac{t_{i-1}+t_i}{2}\right)^2\leq\left(\frac{\tau}{2}\right)^2,\quad \text{for all }t\in\left[t_{i-1},t_i\right].
            \nonumber
        \end{equation}

	Then, we consider boundary terms. For $i=0$, Lemma \ref{lem:truncation_error} and \eqref{equ:boundary_correction}  yield 
	 \begin{equation}
            -D_{tt}e^0-\Delta_g e^0=-\frac{1}{\tau}G^1.
            \label{proof:lem:temporal_boundary_term}
        \end{equation}  
        
    Taking the inner product of \eqref{proof:lem:temporal_boundary_term} with $\tau e^0$ over $\mathcal{M}$, we obtain
        \begin{equation}
            \begin{aligned}
                -\tau\left(D_{tt}e^0,e^0\right)_{\mathcal{M}} +\tau\left\|\nabla_g e^0\right\|_{L^2(\mathcal{M})}^2+\left(G^1,e^0\right)_{\mathcal{M}}=0.
            \end{aligned}
            \label{proof:lem:temporal_boundary_term_inner_product}
        \end{equation}

        Analogously, by applying Lemma \ref{lem:truncation_error} and \eqref{equ:boundary_correction}, the argument for the case $i = N$ can be written as follows:
        \begin{equation}
           \begin{aligned}
               &-\tau\left(D_{tt}e^{N},e^{N}\right)_{\mathcal{M}} +\tau\left\|\nabla_g e^{N}\right\|_{L^2(\mathcal{M})}^2-(G^{N},e^{N})_{\mathcal{M}}\\
               =&2\left(\int_{t_{N-1}}^{t_{N}}\left(t-\frac{t_{N-1}+t_{N}}{2}\right)\partial_{ttt}u(t)\,\mathrm{d}t,e^{N}\right)_{\mathcal{M}}\\
               \leq&\tau\left\|\nabla_g e^{N}\right\|_{L^2(\mathcal{M})}^2+\tau^4\int_{t_{{N}-1}}^{t_{N}}\left(\left\|\nabla_g\partial_{tt}u\right\|_{L^2(\mathcal{M})}^2+\left\|\nabla_g\Delta_g^{-1}\partial_{tt}f\right\|_{L^2(\mathcal{M})}^2\right)\mathrm{d}t.
           \end{aligned}
           \label{proof:lem:temporal_boundary_term_inner_product_2}
       \end{equation}
       
       Summing \eqref{relations_between_E_and_G_3} over $i = 1$ to $i = N - 1$, adding the boundary terms \eqref{proof:lem:temporal_boundary_term_inner_product} and \eqref{proof:lem:temporal_boundary_term_inner_product_2}, and applying the summation by parts identities \eqref{equ:classical_summation_by_parts} and \eqref{equ:summation_by_parts},  we obtain
       \begin{equation}
		\begin{aligned}
			&2\sum_{i=1}^{N} \tau\left\|D_t e^i\right\|_{L^2(\mathcal{M})}^2 +\sum_{i=0}^{N}\tau \left\|\nabla_g e^i\right\|_{L^2(\mathcal{M})}^2-\sum_{i=1}^{N}\tau \left(G^i,D_t e^i\right)_{L^2(\mathcal{M})}\\
			\leq &\tau^4\left( \left\|\nabla_g \partial_{tt}u\right\|_{L^2(\mathcal{T}; L^2(\mathcal{M}))}^2+\left\|\nabla_g\Delta_g^{-1}\partial_{tt}f\right\|_{L^2(\mathcal{T}; L^2(\mathcal{M}))}^2 \right).
		\end{aligned}
		\label{relations_between_E_and_G_4}
	\end{equation}
	
	Applying Young’s inequality to the inner product term $\left(G^i, D_t e^i\right)_{L^2(\mathcal{M})}$ in \eqref{relations_between_E_and_G_4}, we deduce that
		\begin{equation}
		\begin{aligned}
			&\sum_{i=1}^{N} \tau\left\|D_t e^i\right\|_{L^2(\mathcal{M})}^2 +\sum_{i=0}^{N}\tau \left\|\nabla_g e^i\right\|_{L^2(\mathcal{M})}^2\\
            \leq&\tau^4\left( \left\|\nabla_g \partial_{tt}u\right\|_{L^2(\mathcal{T}; L^2(\mathcal{M}))}^2 +\left\|\nabla_g\Delta_g^{-1}\partial_{tt}f\right\|_{L^2(\mathcal{T}; L^2(\mathcal{M}))}^2  \right)+\sum_{i=1}^{N}\tau\left\|G^i\right\|_{L^2(\mathcal{M})}^2\\
			\leq&c\tau^4\left( \left\|\nabla_g \partial_{tt}u\right\|_{L^2(\mathcal{T}; L^2(\mathcal{M}))}^2 +\left\|\nabla_g\Delta_g^{-1}\partial_{tt}f\right\|_{L^2(\mathcal{T}; L^2(\mathcal{M}))}^2 +\left\|\partial_{ttt}u\right\|_{L^2(\mathcal{T}; L^2(\mathcal{M}))}^2 \right).
		\end{aligned}
		\label{relations_between_E_and_G_5}
	\end{equation}
	
 Notice that 
 \begin{equation}
 	\begin{aligned}
 		&\left\|\nabla_g(I_{\tau}u-u_N)\right\|_{L^2(\mathcal{T}; L^2(\mathcal{M}))}^2 + \left\|\partial_t(I_{\tau}u-u_N)\right\|_{L^2(\mathcal{T}; L^2(\mathcal{M}))}^2 \\
 		= &\sum_{i=1}^{N}\int_{t_{i-1}}^{t_i}\left\|\nabla_g(I_{\tau}u-u_N)\right\|_{L^2(\mathcal{M})}^2\,\mathrm{d}t + \sum_{i=1}^{N}\int_{t_{i-1}}^{t_i}\left\|\partial_t(I_{\tau}u-u_N)\right\|_{L^2(\mathcal{M})}^2\,\mathrm{d}t\\
 			\leq&c\tau^4\left( \left\|\nabla_g \partial_{tt}u\right\|_{L^2(\mathcal{T};L^2(\mathcal{M}))}^2+\left\|\nabla_g\Delta_g^{-1}\partial_{tt}f\right\|_{L^2(\mathcal{T};L^2(\mathcal{M}))}^2+\left\|\partial_{ttt}u\right\|_{L^2(\mathcal{T};L^2(\mathcal{M}))}^2\right),
 	\end{aligned}
 \end{equation}
which concludes the proof.
\end{proof}

\subsection{Spatial supercloseness}\label{ssec:spa}
Let $I_h$ be the standard Lagrange interpolation operator from $C(\mathcal{M}_h)$ into $S_h$. From the standard surface finite element theory \cite{brenner2008fembook, ciarlet2002fembook, dziuk1988sfem}, it follows that 
\begin{theorem}    \label{thm:spatial_interpoltation_approximation}
    For any $\omega\in H^2(\mathcal{M})$, the interpolation operator $I_h$ satisfies the following properties:
    \begin{equation}
        \left\|T_h\omega-I_hT_h\omega\right\|_{L^2(\mathcal{M}_h)}+h\left\|\nabla_{g_h}\left(T_h\omega-I_hT_h\omega\right)\right\|_{L^2(\mathcal{M}_h)}\lesssim h^2\left\| \omega\right\|_{H^2(\mathcal{M})}.
    \end{equation}
\end{theorem}

For the interpolation operator $I_h$, \cite{wei2010sprsurface} also establishes the following supercloseness results.
\begin{lemma}
	Suppose the discrete surface $\mathcal{M}_h$ satisfies the $\mathcal{O}(h^{2\sigma})$ irregular condition.  If $\omega\in H^3(\mathcal{M})\cap W_{\infty}^2(\mathcal{M})$, then for all $v_h\in S_h$, we have
	\label{Lemma:integral_T_h_u-I_h_T_h_u}
	\begin{equation}
		\int_{\mathcal{M}_h}\nabla_{g_h}\left(T_h \omega-I_hT_h \omega\right)\cdot\nabla_{g_h}v_h\,\mathrm{d}\sigma_h\lesssim h^{1+\min\{1,\sigma\}}\left(\norm{\omega}_{H^3(\mathcal{M})}+\norm{\omega}_{W^{2}_{\infty}(\mathcal{M})}\right)\left|v_h\right|_{H^1(\mathcal{M}_h)}.
		\label{integral_T_h_u-I_h_T_h_u}
	\end{equation}
\end{lemma}

Before we present the main supercloseness result, we need the following lemma about the spatial truncation error.
\begin{lemma}\label{lem:intermediate_result}
	Suppose the discrete surface $\mathcal{M}_h$ satisfies the $\mathcal{O}(h^{2\sigma})$ irregular condition. Let $u^i\in H^3(\mathcal{M})\cap W_{\infty}^2(\mathcal{M})$ be the solution of temporal semi-discretization \eqref{equ:semi_discretization} and $u^i_h$ be the solution of full-discretization \eqref{equ:full_discrete}. Then, we have
		\begin{equation} \label{equ:intermediate_result}
        	\begin{aligned}
			  &-\int_{\mathcal{M}_h}D_{tt}(u_h^i-I_h T_h u^i)\cdot v_h\,\mathrm{d}\sigma_{g_h}+\int_{\mathcal{M}_h}\nabla_{g_h}(u_h^i-I_h T_h u^i)\cdot\nabla_{g_h}v_h\,\mathrm{d}\sigma_{g_h}\\
			\lesssim & \; h^{1+\min\{1,\sigma\}}{\left( \left\|u^i\right\|_{H^3(\mathcal{M})}+\left\|u^i\right\|_{W^{2}_{\infty}(\mathcal{M})}\right)
			}\left\|\nabla_{g_h}v_h\right\|_{L^2(\mathcal{M}_h)}-\int_{\mathcal{M}_h}D_{tt}(T_hu^i-I_h T_h u^i) v_h\,\mathrm{d}\sigma_{g_h},
		\end{aligned}
	\end{equation}
	for all $v_h\in S_h$ and $i=0, \cdots, N$. 
\end{lemma}

\begin{proof}
	To establish \eqref{equ:intermediate_result}, we rewrite the left hand side of  \eqref{equ:intermediate_result} to get
	\begin{align}
			  &-\int_{\mathcal{M}_h}D_{tt}(u_h^i-I_h T_h u^i) v_h\,\mathrm{d}\sigma_{g_h}+\int_{\mathcal{M}_h}\nabla_{g_h}(u_h^i-I_h T_h u^i)\cdot\nabla_{g_h}v_h\,\mathrm{d}\sigma_{g_h}\nonumber \\
		=   &-\int_{\mathcal{M}_h}(D_{tt}u_h^i-D_{tt} T_hu^i) v_h\,\mathrm{d}\sigma_{g_h}+\int_{\mathcal{M}_h}(\nabla_{g_h}u_h^i -\nabla_{g_h}T_hu^i)\cdot\nabla_{g_h}v_h\,\mathrm{d}\sigma_{g_h} \label{form:equation}\\
			  &-\int_{\mathcal{M}_h}D_{tt}(T_hu^i-I_h T_h u^i) v_h\,\mathrm{d}\sigma_{g_h}+\int_{\mathcal{M}_h}\nabla_{g_h}(T_hu^i-I_h T_h u^i)\cdot\nabla_{g_h}v_h\,\mathrm{d}\sigma_{g_h} \label{form:closeness}.
	\end{align}
    
Let us have a close look to the two terms in the two lines \eqref{form:equation} and \eqref{form:closeness} separately.

For \eqref{form:equation}, we have 
	\begin{align}
   & -\int_{\mathcal{M}_h}(D_{tt}u_h^i-D_{tt} T_hu^i) v_h\,\mathrm{d}\sigma_{g_h}+\int_{\mathcal{M}_h}(\nabla_{g_h}u_h^i -\nabla_{g_h}T_hu^i)\cdot\nabla_{g_h}v_h\,\mathrm{d}\sigma_{g_h}\nonumber
   \\
=&\int_{\mathcal{M}_h} (-D_{tt}u_h^i v_h+\nabla_{g_h}u_h^i \cdot \nabla_{g_h} v_h )\,\mathrm{d}\sigma_{g_h}
    -\int_{\mathcal{M}_h}( -D_{tt} T_hu^i v_h +\nabla_{g_h}T_hu^i\cdot \nabla_{g_h}v_h)\,\mathrm{d}\sigma_{g_h}. 
    \label{eq:keyeqn}
    \end{align}
    
The first term on the right hand side of Equation \eqref{eq:keyeqn} can be written equivalently as follows using the finite element formulation of equation \eqref{equ:semi_discretization_weak}
    \begin{equation*}
\int_{\mathcal{M}_h}(-D_{tt}u_h^i v_h+\nabla_{g_h}u_h^i \cdot \nabla_{g_h} v_h )\,\mathrm{d}\sigma_{g_h}= \int_{\mathcal{M}_h}T_h(f(t_i)+b^i) v_h\,\mathrm{d}\sigma_{g_h}.
    \end{equation*}
    
Now, taking into account the product in the form of \eqref{eq:Theorem:Integral_of_gradients_transform}, we estimate the second term on the right hand side of Equation \eqref{eq:keyeqn}.
\begin{align*}
 & \int_{\mathcal{M}_h}( -D_{tt} T_hu^i v_h +\nabla_{g_h}T_hu^i\cdot \nabla_{g_h}v_h)\,\mathrm{d}\sigma_{g_h}\\
 =&
 \int_{\mathcal{M}_h} T_h(-D_{tt} u^i -\Delta_{g}u^i)   v_h\,\mathrm{d}\sigma_{g_h}
 -\int_{\mathcal{M}_h} (-T_h\Delta_{g}u^i v_h -\nabla _{g_h} T_h u^i\cdot \nabla _{g_h} v_h)\,\mathrm{d}\sigma_{g_h}\\
  =&
 \int_{\mathcal{M}_h} T_h(f(t_i)+b^i) v_h \,\mathrm{d}\sigma_{g_h}
 - \int_{\mathcal{M}} (-\Delta_{g}u^i\, T_h^{-1}v_h - T_h^{-1}(\nabla _{g_h} T_h u^i\cdot \nabla _{g_h} v_h))\,\frac{\sqrt{|g_h|}}{\sqrt{|g|}}\mathrm{d}\sigma_{g} \\
\geq &\int_{\mathcal{M}_h} T_h(f(t_i)+b^i) v_h \,\mathrm{d}\sigma_{g_h} - \norm{\frac{\sqrt{|g_h|}}{\sqrt{|g|}}}_{L^{\infty}} \left|\int_{\mathcal{M}} (-\Delta_{g}u^i\, T_h^{-1}v_h - T_h^{-1}(\nabla _{g_h} T_h u^i\cdot \nabla _{g_h} v_h))\,\mathrm{d}\sigma_{g}\right| \\
\geq &\int_{\mathcal{M}_h} T_h(f(t_i)+b^i) v_h \,\mathrm{d}\sigma_{g_h}
-  \norm{\frac{\sqrt{|g_h|}}{\sqrt{|g|}}}_{L^{\infty}}\norm{g(g^{-1}-g_h^{-1})}_{L^{\infty}} \left|\int_{\mathcal{M}}  \nabla_g u^i \cdot \nabla_g T_h^{-1}v_h \,\mathrm{d}\sigma_{g}\right|.
\end{align*}

We combine all these estimates and go back to \eqref{eq:keyeqn}. Using Theorem \ref{thm:metrictensor_approximation} and the boundedness of $T_h$ \eqref{equ:boundedness_of_T_h}, we have
\begin{align*}
&-\int_{\mathcal{M}_h}(D_{tt}u_h^i-D_{tt} T_hu^i) v_h\,\mathrm{d}\sigma_{g_h}+\int_{\mathcal{M}_h}(\nabla_{g_h}u_h^i -\nabla_{g_h}T_hu^i)\cdot\nabla_{g_h}v_h\,\mathrm{d}\sigma_{g_h}\\
\leq & \norm{\frac{\sqrt{g_h}}{\sqrt{g}}}_{L^{\infty}}\norm{g(g^{-1}-g_h^{-1})}_{L^{\infty}} \left|\int_{\mathcal{M}}  \nabla_g u^i \cdot \nabla_g T_h^{-1}v_h \,\mathrm{d}\sigma_{g}\right|\\
\lesssim & h^2 \left|u^i\right|_{H^1(\mathcal{M})}\left|T_h^{-1}v_h\right|_{H^1(\mathcal{M})}
\lesssim h^2|u^i|_{H^1(\mathcal{M})}|v_h|_{H^1(\mathcal{M}_h)}.
\end{align*}

The second term in \eqref{form:closeness} is estimated using \eqref{integral_T_h_u-I_h_T_h_u} in Lemma \ref{lem:intermediate_result} directly. 
\begin{align*}
 &\int_{\mathcal{M}_h}\nabla_{g_h}(T_hu^i-I_h T_h u^i)\cdot\nabla_{g_h}v_h\,\mathrm{d}\sigma_{g_h} 
\lesssim  h^{1+\min\{1,\sigma\}}\left( \norm{u^i}_{H^3(\mathcal{M})}+\norm{u^i}_{W^{2}_{\infty}(\mathcal{M})}\right)\left\|\nabla_{g_h}v_h\right\|_{L^2(\mathcal{M}_h)}.
\end{align*}

Summarizing all the estimates we have the conclusion for all $i=0,1,\cdots,N$.

\end{proof}

\begin{remark}
 Using a general regularity result for second-order elliptic operators (cf. \cite{Aubin82,evans2010pdebook, grigoryan2009heat}) and an induction argument, we obtain that for $k\geq 2$ and $ p\in\{2,\infty\} $, if $f(t_i),\, \mu_0,\, \mu_1\in W^{k-2}_{p}(\mathcal{M})$, then $u^i\in W^{k}_{p}(\mathcal{M})$ for $i=1,\ldots,N$. Consequently, the condition in Lemma~\ref{lem:intermediate_result} that $u^i\in H^3(\mathcal{M})\cap W_{\infty}^2(\mathcal{M})$ can be satisfied by requiring, e.g., 
    $f(t_i),\mu_0,\, \mu_1\in H^1(\mathcal{M})\cap L^{\infty}(\mathcal{M})$.
On the other hand, in the following Theorems, if we further assume $u\in H^4(\mathcal{T};H^2(\mathcal{M}))\cap H^2(\mathcal{T};H^3(\mathcal{M})\cap W^{2}_{\infty}(\mathcal{M}))$, then the regularity of $u^i\in H^3(\mathcal{M})\cap W_{\infty}^2(\mathcal{M})$ will be assured by the indicated regularity passing from $u$ to $f$ and $\mu_0,\, \mu_1$. 

\end{remark}

With the above preparation, we are now ready to present our main supercloseness results of FDM-sFEM on the spatial manifold.
\begin{theorem}	\label{thm:spatial_supercloseness}
	Suppose the discrete surface $\mathcal{M}_h$ satisfies the $\mathcal{O}(h^{2\sigma})$ irregular condition. Let $u_N$ be the semi-discrete FDM solution defined in \eqref{equ:semi-discrete_solution} with $u^i\in H^3(\mathcal{M})\cap W_{\infty}^2(\mathcal{M})$ and $u_h$ be the FDM-sFEM solution defined in \eqref{equ:full_discrete}. Then, we have 	
		\begin{equation} \label{equ:spatial_supercloseness}
		\begin{aligned}
			&\left\|\nabla_{g_h}\left(u_h-I_hT_hu_N\right)\right\|_{L^2(\mathcal{T};L^2(\mathcal{M}))}^2
			+ \left\|\partial_{t}\left(u_h-I_hT_hu_N\right)\right\|_{L^2(\mathcal{T};L^2(\mathcal{M}))}^2 \\
			\lesssim&h^{2+2\min\{1,\sigma\}}\left(\left\|u_N\right\|_{L^2(\mathcal{T};H^3(\mathcal{M}))}^2+\left\|u_N\right\|_{L^2(\mathcal{T};W^{2}_{\infty}(\mathcal{M}))}^2\right)
			+h^4\left\|\partial_tu_N\right\|_{L^2(\mathcal{T};H^2(\mathcal{M}))}^2.
		\end{aligned}
	\end{equation}
\end{theorem} 

\begin{proof}
    To show \eqref{equ:spatial_supercloseness}, we sum \eqref{equ:intermediate_result} for $i = 1, \ldots, N-1$ with $ v_h = 2\tau\left(u_h^i - I_h T_h u^i\right) $  and include the boundary contributions at \( i = 0 \) and \( i = N \) with \( v_h = \frac{1}{2} \tau \left(u_h^i - I_h T_h u^i\right) \). Then, by applying the summation by parts formula \eqref{equ:summation_by_parts}, we can deduce that
    	\begin{equation}
		\begin{aligned}
			&2\sum_{i=1}^{N}\tau\left\|D_t\left(u_h^i-I_hT_hu^i\right)\right\|_{L^2(\mathcal{M}_h)}^2+2\sum_{i=0}^{N}\omega_i\tau\left\|\nabla_{g_h}\left(u_h^i-I_hT_h u^i\right)\right\|_{L^2(\mathcal{M}_h)}^2\\
			\lesssim&  2\sum_{i=0}^{N}\omega_i\tau h^{1+\min\{1,\sigma\}}{\left( \left\|u^i\right\|_{H^3(\mathcal{M})}+\left\|u^i\right\|_{W^{2}_{\infty}(\mathcal{M})}\right)
			}\left\|\nabla_{g_h}\left(u_h^i-I_h T_h u^i\right)\right\|_{L^2(\mathcal{M}_h)}\\
            &+2\sum_{i=1}^{N}\tau\int_{\mathcal{M}_h}D_{t}(T_hu^i-I_h T_h u^i)\cdot D_{t}(u_h^i-I_hT_h u^i)\,\mathrm{d}\sigma_{g_h}\\
            \leq&  h^{2+2\min\{1,\sigma\}}\sum_{i=0}^{N}\omega_i\tau{\left( \left\|u^i\right\|_{H^3(\mathcal{M})}^2+\left\|u^i\right\|_{W^{2}_{\infty}(\mathcal{M})}^2\right)
			}+\sum_{i=0}^{N}\omega_i \tau \left\|\nabla_{g_h}\left(u_h^i-I_hT_hu^i\right)\right\|_{L^2(\mathcal{M}_h)}^2\\
            &+\sum_{i=1}^{N}\tau \left\|D_t\left(T_h u^i-I_hT_h u^i\right)\right\|_{L^2(\mathcal{M}_h)}^2+\sum_{i=1}^{N}\tau\left\|D_t\left(u_h^i-I_hT_hu^i\right)\right\|_{L^2(\mathcal{M}_h)}^2\\
            \lesssim&   h^{2+2\min\{1,\sigma\}}\sum_{i=0}^{N}\omega_i\tau{\left( \left\|u^i\right\|_{H^3(\mathcal{M})}^2+\left\|u^i\right\|_{W^{2}_{\infty}(\mathcal{M})}^2\right)
			}+\sum_{i=0}^{N}\omega_i \tau \left\|\nabla_{g_h}\left(u_h^i-I_hT_hu^i\right)\right\|_{L^2(\mathcal{M}_h)}^2\\
            &+h^4\sum_{i=1}^{N}\tau \left\|D_t u^i\right\|_{H^2(\mathcal{M})}^2+\sum_{i=1}^{N}\tau\left\|D_t\left(u_h^i-I_hT_hu^i\right)\right\|_{L^2(\mathcal{M}_h)}^2,\\
		\end{aligned}
            \label{Proof:Theoremm:gradient_u_h-u_I_4-0}
	\end{equation}
        where the weights $\omega_i=1$ for $i=1,\ldots,N-1$ and $\omega_0=\omega_{N}=\frac{1}{2}$. The second last inequality follows from Young's inequality, whereas the last inequality derives from Theorem \ref{thm:spatial_interpoltation_approximation} and the boundedness of $T_h$ \eqref{equ:boundedness_of_T_h}.

        By transposing the terms of \eqref{Proof:Theoremm:gradient_u_h-u_I_4-0} and with the definition of $u_N$ \eqref{equ:semi-discrete_solution}, we obtain the following estimate
        \begin{equation}
            \begin{aligned}
                &\sum_{i=1}^{N}\tau\left\|D_t\left(u_h^i-I_hT_hu^i\right)\right\|_{L^2(\mathcal{M}_h)}^2+\sum_{i=0}^{N}\omega_i\tau\left\|\nabla_{g_h}\left(u_h^i-I_hT_h u^i\right)\right\|_{L^2(\mathcal{M}_h)}^2\\
                \lesssim &  h^{2+2\min\{1,\sigma\}}\left( \left\|u_N\right\|_{L^2(\mathcal{T};H^3(\mathcal{M}))}^2+\left\|u_N\right\|_{L^2(\mathcal{T};W^{2}_{\infty}(\mathcal{M}))}^2\right)+h^4\left\|\partial_t u_N\right\|_{L^2(\mathcal{T};H^2(\mathcal{M}))}^2.
            \end{aligned}
            \label{Proof:Theoremm:gradient_u_h-u_I_4}
        \end{equation}

    From \eqref{Proof:Theoremm:gradient_u_h-u_I_4}, using the definitions of $u_N$ and $u_h$ \eqref{equ:semi-discrete_solution}, \eqref{equ:full_discrete_solution}, we deduce
        \begin{equation}
            \begin{aligned}
                &\left\|\nabla_{g_h}\left(u_h-I_hT_hu_N\right)\right\|_{L^2(\mathcal{T};L^2(\mathcal{M}))}^2
                =\sum_{i=1}^{N}\int_{t_{i-1}}^{t_i}	\left\|\nabla_{g_h}\left(u_h^i-I_hT_hu^i\right)\right\|_{L^2(\mathcal{M}_h)}^2\,\mathrm{d}t\\
			=&\sum_{i=1}^{N}\int_{t_{i-1}}^{t_i}	\left\|\nabla_{g_h}\left(u_h^{i-1}-I_hT_hu^{i-1}\right)\phi_{i-1}(t)+\nabla_{g_h}\left(u_h^{i}-I_hT_hu^{i}\right)\phi_{i}(t)\right\|_{L^2(\mathcal{M}_h)}^2\,\mathrm{d}t\\
			\lesssim&\sum_{i=0}^{N}\tau\left\|\nabla_{g_h}\left(u_h^i-I_hT_h u^i\right)\right\|_{L^2(\mathcal{M}_h)}^2\\
			\lesssim&h^{2+2\min\{1,\sigma\}}\left( \left\|u_N\right\|_{L^2(\mathcal{T};H^3(\mathcal{M}))}^2+\left\|u_N\right\|_{L^2(\mathcal{T};W^{2}_{\infty}(\mathcal{M}))}^2\right)+h^4\left\|\partial_t u_N\right\|_{L^2(\mathcal{T};H^2(\mathcal{M}))}^2.
            \end{aligned}
            \nonumber
        \end{equation}

Similarly, from \eqref{Proof:Theoremm:gradient_u_h-u_I_4} it follows that
	\begin{equation}
		\begin{aligned}
			&\left\|\partial_{t}\left(u_h-I_hT_hu_N\right)\right\|_{L^2(\mathcal{T};L^2(\mathcal{M}))}^2
            =\sum_{i=1}^{N}\int_{t_{i-1}}^{t_i}\left\|\partial_{t}\left(u_h^i-I_hT_hu^i\right)\right\|_{L^2(\mathcal{M}_h)}^2\,\mathrm{d}t\\
			=&\sum_{i=1}^{N}\tau\left\|D_{t}\left(u_h^i-I_hT_hu^i\right)\right\|_{L^2(\mathcal{M}_h)}^2\\
			\lesssim&h^{2+2\min\{1,\sigma\}}\left( \left\|u_N\right\|_{L^2(\mathcal{T};H^3(\mathcal{M}))}^2+\left\|u_N\right\|_{L^2(\mathcal{T};W^{2}_{\infty}(\mathcal{M}))}^2\right)+h^4\left\|\partial_t u_N\right\|_{L^2(\mathcal{T};H^2(\mathcal{M}))}^2.
		\end{aligned}
        \nonumber
	\end{equation}
	This completes the proof. 
\end{proof}

\subsection{Superconvergence of the recovered gradient}\label{ssec:srg}
In this subsection, we establish the superconvergence properties of the recovered gradient via decomposition into temporal and spatial components.

\subsubsection{Temporal superconvergence analysis}
We begin by analyzing the error between the exact temporal derivative and the PPR recovered temporal gradient obtained from the finite difference method (FDM) solution
\begin{theorem}\label{Theorem:G_tau_u_N}
	Let $u\in H^4(\mathcal{T}; L^2(\mathcal{M}))\cap H^2(\mathcal{T};H^2(\mathcal{M}))$ denote the exact solution to \eqref{equ:model}, $u_N$ be the semi-discrete FDM solution defined in \eqref{equ:semi-discrete_solution}, and let $G_{\tau}$ denote the temporal PPR operator. Then the recovered gradient satisfies
\begin{equation}        \label{equ:G_tau_u_N}
		\|\partial_t u-G_{\tau}u_N\|_{L^2(\mathcal{T};L^2(\mathcal{M}))}^2\lesssim \tau^4\left( \left\|\nabla_g\Delta_g^{-1}\partial_{tt}f\right\|_{L^2(\mathcal{T};L^2(\mathcal{M}))}^2+ \|u\|_{H^3(\mathcal{T};L^2(\mathcal{M}))}^2 \right).
	\end{equation}
\end{theorem}
\begin{proof}
We decompose the term $\partial_t u - G_{\tau} u_N$ as
	\begin{equation}
		\partial_t u-G_{\tau} u_N=\underbrace{\partial_t u-G_{\tau}I_{\tau}u}_{I_1}+\underbrace{G_{\tau}I_{\tau}u-G_{\tau}u_N}_{I_2}.
        \label{G_tau_decomposition}
	\end{equation}
	
	We first estimate the term $I_2$. By invoking Lemma~\ref{lem:temporal} and Lemma~\ref{lem:G_tau_boundedness}, we obtain
		\begin{equation}
		\begin{aligned}
			&\|G_{\tau}I_{\tau}u-G_{\tau}u_N\|_{L^2(\mathcal{T};L^2(\mathcal{M}))}^2
			=\|G_{\tau}(I_{\tau}u-u_N)\|_{L^2(\mathcal{T};L^2(\mathcal{M}))}^2\\
			\lesssim& \|\partial_t(I_{\tau}u- u_N)\|_{L^2(\mathcal{T};L^2(\mathcal{M}))}^2\\
			\lesssim& \tau^4\left( \left\|\nabla_g \partial_{tt}u\right\|_{L^2(\mathcal{T};L^2(\mathcal{M}))}^2+\left\|\nabla_g\Delta_g^{-1}\partial_{tt}f\right\|_{L^2(\mathcal{T};L^2(\mathcal{M}))}^2+\left\|\partial_{ttt}u\right\|_{L^2(\mathcal{T};L^2(\mathcal{M}))}^2\right).
		\end{aligned}
		\label{G_tau_decomposition_latter}
	\end{equation}

For the term $I_1$, the consistency of the temporal PPR gradient recovery operator $G_{\tau}$ (see Lemma~\ref{lem:PPR_consistency}) implies
\begin{equation}
		\left\|\partial_t u-G_{\tau}I_{\tau}u\right\|_{L^2(\mathcal{T};L^2(\mathcal{M}))}^2\lesssim \tau^4\|u\|_{H^3(\mathcal{T};L^2(\mathcal{M}))}^2.
		\label{G_tau_decomposition_former}
\end{equation}

Finally, estimate \eqref{equ:G_tau_u_N} follows by combining the bounds \eqref{G_tau_decomposition_latter} and \eqref{G_tau_decomposition_former}.
\end{proof}

With the above preparation, we are now ready to present our main superconvergence in temporal direction.
\begin{theorem}    \label{thm:time_superconvergence}
   Suppose the discrete surface $\mathcal{M}_h$ satisfies the $\mathcal{O}(h^{2\sigma})$ irregular condition. Let $u\in H^4(\mathcal{T}; L^2(\mathcal{M}))\cap H^2(\mathcal{T}; H^2(\mathcal{M}))$ denote the exact solution to \eqref{equ:model}, $u_N$ be the semi-discrete FDM solution defined in \eqref{equ:semi-discrete_solution} with $u^i\in H^3(\mathcal{M})\cap W^{2}_{\infty}(\mathcal{M})$, and $u_h$ be the FDM-sFEM solution defined in \eqref{equ:full_discrete}. Then there holds
    \begin{equation}
        \begin{aligned}
            &\left\|\partial_t u-G_{\tau} T_h^{-1}u_h\right\|_{L^2(\mathcal{T};L^2(\mathcal{M}))}^2\\
            \lesssim& \tau^4\left( \left\|\nabla_g\Delta_g^{-1}\partial_{tt}f\right\|_{L^2(\mathcal{T};L^2(\mathcal{M}))}^2+ \|u\|_{H^3(\mathcal{T};L^2(\mathcal{M}))}^2 \right)+h^4\left\|\partial_t u_N\right\|_{L^2(\mathcal{T};H^2(\mathcal{M}))}^2\\
            &+h^{2+2\min\{1,\sigma\}}\left(\left\|u_N\right\|_{L^2(\mathcal{T};H^3(\mathcal{M}))}^2+\left\|u_N\right\|_{L^2(\mathcal{T};W^{2}_{\infty}(\mathcal{M}))}^2\right).
        \end{aligned}
        \label{eq:G_tau_u_h}
    \end{equation}
\end{theorem}

\begin{proof}
We decompose the error as
    \begin{equation*}
        \begin{aligned}
            \partial_t u-G_{\tau} T_h^{-1} u_h=\underbrace{\left(\partial_t u -G_{\tau}u_N\right)}_{I_1}+\underbrace{T_h^{-1}\left(G_{\tau}T_hu_N-G_{\tau}I_hT_hu_N\right)}_{I_2}+\underbrace{T_h^{-1}\left(G_{\tau}I_hT_hu_N-G_{\tau}u_h\right)}_{I_3}.
        \end{aligned}
    \end{equation*}
    
The term $I_1$ has already been estimated in Theorem~\ref{Theorem:G_tau_u_N}. 
To estimate $I_2$, we recall that $I_h$ denotes the standard Lagrange interpolation operator into the surface finite element space $S_h$. By using Lemma~\ref{lem:boundedness_of_T_h}, Lemma~\ref{lem:G_tau_boundedness}, and Theorem~\ref{thm:spatial_interpoltation_approximation}, we obtain
    \begin{equation}
        \begin{aligned}
            &\left\|T_h^{-1}\left(G_{\tau}T_hu_N-G_{\tau}I_hT_hu_N\right)\right\|_{L^2(\mathcal{T};L^2(\mathcal{M}))}=\left\|T_h^{-1}G_{\tau}\left(T_hu_N-I_hT_h u_N\right)\right\|_{L^2(\mathcal{T};L^2(\mathcal{M}))}\\
            \lesssim& \left\|\partial_t\left(T_hu_N-I_hT_h u_N\right)\right\|_{L^2(\mathcal{T};L^2(\mathcal{M}))}
            =\left\|T_h\partial_tu_N-I_hT_h \partial_tu_N\right\|_{L^2(\mathcal{T};L^2(\mathcal{M}))}\\
            \lesssim&h^2\left\|\partial_tu_N\right\|_{L^2(\mathcal{T};H^2(\mathcal{M}))}.
        \end{aligned}
        \label{proof:G_tau_u_h_e1}
    \end{equation} 
    
To estimate $I_3$, we employ the boundedness of the PPR gradient recovery operator $G_{\tau}$ from Lemma~\ref{lem:G_tau_boundedness} and the spatial supercloseness result from Theorem~\ref{thm:spatial_supercloseness}, which yield
    \begin{equation}
        \begin{aligned}
           & \left\|T_h^{-1}\left(G_{\tau}I_hT_hu_N-G_{\tau}u_h\right)\right\|_{L^2(\mathcal{T};L^2(\mathcal{M}))}^2=\left\|T_h^{-1}G_{\tau}(I_h T_hu_N-u_h)\right\|_{L^2(\mathcal{T};L^2(\mathcal{M}))}^2\\
            \lesssim& \left\|\partial_t\left(I_hT_hu_N-u_h\right)\right\|_{L^2(\mathcal{T};L^2(\mathcal{M}))}^22\\
            \lesssim&h^{2+2\min\{1,\sigma\}}\left(\left\|u_N\right\|_{L^2(T;H^3(\mathcal{M}))}^2+\left\|u_N\right\|_{L^2(T;W^{2}_{\infty}(\mathcal{M}))}^2\right)+h^4\left\|\partial_t u_N\right\|_{L^2(\mathcal{T};H^2(\mathcal{M}))}^2.
        \end{aligned}
        \label{proof:G_tau_u_h_e2}
    \end{equation}
    
Combining the estimates \eqref{proof:G_tau_u_h_e1} and \eqref{proof:G_tau_u_h_e2} with the result of Theorem~\ref{Theorem:G_tau_u_N} yields the desired estimate \eqref{eq:G_tau_u_h}, thus completing the proof.
\end{proof}

\subsubsection{Spatial superconvergence analysis}
In this subsection, we establish the superconvergence of the recovered surface gradient using the PPPR gradient recovery operator.  We begin with the following superconvergence results.

\begin{theorem}\label{thm:spatial_semi_superconvergence}
     Suppose the discrete surface $\mathcal{M}_h$ satisfies the $\mathcal{O}(h^{2\sigma})$ irregular condition. Let $u_N$ be the semi-discrete FDM solution defined in \eqref{equ:semi-discrete_solution} with $u^i\in H^3(\mathcal{M})\cap W_{\infty}^{3}(\mathcal{M})$ and $u_h$ be the FDM-sFEM solution defined in \eqref{equ:full_discrete}. Then there holds
    \begin{equation}\label{equ:semi_superconvergence}
        \begin{aligned}
            \left\|\nabla_g u_N-T_h^{-1}G_hu_h\right\|_{L^2(\mathcal{T};L^{2}(\mathcal{M}))}^2
            \lesssim&  h^4\left(\left\|u_N\right\|_{L^2(\mathcal{T};W^{3}_{\infty}(\mathcal{M}))}^2+\left\|\partial_t u_N\right\|_{L^2(\mathcal{T};H^2(\mathcal{M}))}^2\right)\\
            &+h^{2+2\min\{1,\sigma\}}\left(\left\|u_N\right\|_{L^2(\mathcal{T};H^{3}(\mathcal{M}))}^2+\left\|u_N\right\|_{L^2(T;W^{2}_{\infty}(\mathcal{M}))}^2\right).
        \end{aligned}
    \end{equation}
\end{theorem}
\begin{proof}
We begin by decomposing the error as
    \begin{equation}
        \begin{aligned}
            \nabla_g u_N-T_h^{-1}G_hu_h=\underbrace{\nabla_g u_N-T_h^{-1}G_h\left(I_hT_hu_N\right)}_{I_1}+\underbrace{T_h^{-1}G_h\left(I_hT_hu_N\right)-T_h^{-1}G_hu_h}_{I_2}.
        \end{aligned}
        \label{Proof:Supercloseness_of_G_h_e1}
    \end{equation}
    
To estimate the first term $I_1$, we employ the consistency of the PPPR gradient recovery operator $G_h$ established in \eqref{equ:pppr_consistency},
    \begin{equation}
        \left\|\nabla_g u^i-T_h^{-1}G_h \left(I_h T_h u^i\right)\right\|_{L^2(\mathcal{M})}\leq h^2\sqrt{\mathcal{A}(\mathcal{M})}D(g,g^{-1})\left\|u^i\right\|_{W^{3}_{\infty}(\mathcal{M})},
    \end{equation}
    for $i=0,\ldots,N$. 
    
    Using the definition of semi-discrete solution $u_N$ \eqref{equ:semi-discrete_solution}, we obtain
    \begin{equation}
        \begin{aligned}
            &\left\|\nabla_g u_N-T_h^{-1}G_h \left(I_h T_h u_N\right)\right\|_{L^2(\mathcal{T};L^2(\mathcal{M}))}^2
            =\sum_{i=1}^{N}\int_{t_{i-1}}^{t_i}\left\|\nabla_g u_N-T_h^{-1}G_h \left(I_h T_h u_N\right)\right\|_{L^2(\mathcal{M})}^2\,\mathrm{d}t\\
            =&\sum_{i=1}^{N}\int_{t_{i-1}}^{t_i}\left\|\left(\nabla_g u^{i-1}-T_h^{-1}G_h \left(I_h T_h u^{i-1}\right)\right)\phi_{i-1}(t)+\left(\nabla_g u^{i}-T_h^{-1}G_h \left(I_h T_h u^{i}\right)\right)\phi_{i}(t)\right\|_{L^2(\mathcal{M})}^2\,\mathrm{d}t\\
            \lesssim&\sum_{i=0}^{N}\tau\left\|\nabla_g u^{i}-T_h^{-1}G_h \left(I_h T_h u^{i}\right)\right\|_{L^2(\mathcal{M})}^2
            \\
            \leq& h^4\mathcal{A}(\mathcal{M})D(g,g^{-1})^2\left\|u_N\right\|_{L^2(T;W^{3}_{\infty}(\mathcal{M}))}^2.
        \end{aligned}
         \label{Proof:Supercloseness_of_G_h_e2}
    \end{equation}
    
For the second term $I_2$, we use the boundedness of the operator $G_h$ together with the spatial supercloseness result in Theorem~\ref{thm:spatial_supercloseness}, which yields
    \begin{equation}
        \begin{aligned}
            &\left\|T_h^{-1}G_h\left(I_hT_hu_N\right)-T_h^{-1}G_hu_h\right\|_{L^2(\mathcal{T};L^2(\mathcal{M}))}^2
            =\left\|T_h^{-1}G_h\left(I_hT_hu_N-u_h\right)\right\|_{L^2(\mathcal{T};L^2(\mathcal{M}))}^2\\
            \lesssim&\left\|\nabla_{g_h}\left(I_hT_hu_N-u_h\right)\right\|_{L^2(\mathcal{T};L^2(\mathcal{M}))}^2\\
            \lesssim&h^{2+2\min\{1,\sigma\}}\left(\left\|u_N\right\|_{L^2(T;H^3(\mathcal{M}))}^2+\left\|u_N\right\|_{L^2(T;W^{2}_{\infty}(\mathcal{M}))}^2\right)+h^4\left\|\partial_t u_N\right\|_{L^2(\mathcal{T};H^2(\mathcal{M}))}^2.
        \end{aligned}
         \label{Proof:Supercloseness_of_G_h_e3}
    \end{equation}
    
Combining the estimates \eqref{Proof:Supercloseness_of_G_h_e2} and \eqref{Proof:Supercloseness_of_G_h_e3} completes the proof of \eqref{equ:semi_superconvergence}.
\end{proof}

We are now in a position to state our main superconvergence result for the recovered surface gradient.
\begin{theorem}    \label{thm:space_superconvergence}
       Suppose the discrete surface $\mathcal{M}_h$ satisfies the $\mathcal{O}(h^{2\sigma})$ irregular condition. Let $u\in H^4(\mathcal{T}; L^2(\mathcal{M}))\cap H^2(\mathcal{T};H^2(\mathcal{M})$ denote the exact solution to \eqref{equ:model}, $u_N$ be the semi-discrete FDM solution defined in \eqref{equ:semi-discrete_solution} with $u^i\in H^3(\mathcal{M})\cap W_{\infty}^{3}(\mathcal{M})$, and $u_h$ be the FDM-sFEM solution defined in \eqref{equ:full_discrete}. Then there holds
    \begin{equation}\label{equ:main_superconvergence_result}
        \begin{aligned}
            &\left\|\nabla_g u-T_h^{-1}G_h u_h \right\|_{L^2(\mathcal{T};L^2(\mathcal{M}))}^2\\
            \lesssim& h^4\left(\left\|u_N\right\|_{L^2(T;W^{3}_{\infty}(\mathcal{M}))}^2+\left\|\partial_t u_N\right\|_{L^2(\mathcal{T};H^2(\mathcal{M}))}^2\right)\\
            &+h^{2+2\min\{1,\sigma\}}\left(\left\|u_N\right\|_{L^2(T;H^3(\mathcal{M}))}^2+\left\|u_N\right\|_{L^2(T;W^{2}_{\infty}(\mathcal{M}))}^2\right)\\
            &+c\tau^4\left( \left\|\nabla_g \partial_{tt}u\right\|_{L^2(\mathcal{T};L^2(\mathcal{M}))}^2+\left\|\nabla_g\Delta_g^{-1}\partial_{tt}f\right\|_{L^2(\mathcal{T};L^2(\mathcal{M}))}^2+\left\|\partial_{ttt}u\right\|_{L^2(\mathcal{T};L^2(\mathcal{M}))}^2\right).
        \end{aligned}
    \end{equation}
\end{theorem}
\begin{proof}
    We decompose $\nabla_gu-T_h^{-1}G_h u_h$ as 
    \begin{equation}
        \begin{aligned}
            \nabla_g u-T_h^{-1}G_h u_h=\underbrace{\nabla_g\left( u- I_{\tau}u\right)}_{I_1}+\underbrace{\nabla_g\left( I_{\tau}u- u_N\right)}_{I_2} +\underbrace{\left(\nabla_g u_N-T_h^{-1}G_h u_h\right)}_{I_3}.
        \end{aligned}
        \label{proof:Theorem:Superconvergence_of_G_h}
    \end{equation}
    
The contribution $I_1$ is bounded using the interpolation estimate in Lemma~\ref{lem:closeness_temporal_interpolation}, while the term $I_2$ is controlled by the temporal supercloseness result in Lemma~\ref{lem:temporal}. The estimate for $I_3$ has already been established in Theorem~\ref{thm:spatial_semi_superconvergence}. Combining the bounds for $I_1$, $I_2$, and $I_3$ yields the desired estimate \eqref{equ:main_superconvergence_result}, thereby completing the proof. 
\end{proof}

\section{Numerical Experiment}\label{sec:num}
In this section, we present a series of numerical experiments to evaluate the performance of the proposed method and to verify the superconvergent behavior of the post-processed gradients. The resulting linear systems are solved using the fast algorithm developed in \cite{dong2025ot}. In the first example, uniform meshes are generated by mapping rectangular grids onto the torus. For the remaining two examples, the initial triangulations are produced using the three-dimensional surface mesh generation module from the Computational Geometry Algorithms Library (CGAL). Finer meshes are obtained via uniform refinement followed by projection onto $\mathcal{M}$. In both cases, the projection map is explicitly available, and we employ a first-order approximation of the projection as described in \cite{demlow2009highorder}. As a result, the mesh vertices do not lie exactly on the surface $\mathcal{M}$ but instead reside within an $\mathcal{O}(h^2)$ neighborhood of it. This setup highlights the relevance and applicability of the proposed geometric error analysis framework in practical scenarios.
For clarity, we introduce the following notation for the errors:
\begin{equation*}
    \begin{aligned}
        &e=\left\|u_h-T_hu\right\|_{L^2(\mathcal{T}; L^2(\mathcal{M}_h))}, \quad\quad \quad\quad
         De=\left\|\nabla_{p_h}u_h-T_h\nabla_{p} u\right\|_{L^2(\mathcal{T}; L^2(\mathcal{M}_h))},\\
        &De^{T}_{2}=\left\|T_h\partial_t u-G_{\tau}u_h\right\|_{L^2(\mathcal{T}; L^2(\mathcal{M}_h))}, \quad 
        De^{\mathcal{M}}_{2}=\left\|T_h\nabla_gu-G_hu_h\right\|_{L^2(\mathcal{T}; L^2(\mathcal{M}_h))},
    \end{aligned}
\end{equation*}
where $\nabla_{\mathbf{p}} = (\partial_t, \nabla_g)$ denote the time-space derivative and $\nabla_{\mathbf{p_h}}$ is its discrete counterpart.

In our computations, convergence rates are evaluated with respect to the square root of the number of vertices, denoted by $N_v$. We also note that the time step size $\tau$ is halved at each level of mesh refinement, so that $h \approx \tau$ holds throughout the experiments.

\subsection{Numerical example 1}

In this example, we consider the case where the manifold $\mathcal{M}$ is a torus, defined implicitly via the level set function
\begin{equation}
\begin{aligned}
\varphi(x)=\sqrt{\left(\sqrt{x_1^2+x_2^2}-4\right)^2+x_3^2}-1.
\end{aligned}
\end{equation}
We consider a benchmark example of \eqref{equ:model} with exact solution to be prescribed by
\begin{equation*}
u(t,x) = x_1 x_2 e^t.
\end{equation*}
The corresponding source term $f$ and Neumann boundary condition are derived directly from the exact solution $u$.

\begin{table}[H]
    \centering
        \caption{Numerical results of time-space Laplace-Beltrami type equation for Numerical Example 1.}
        		\resizebox{\textwidth}{!}{
    \begin{tabular}{|c|c|c|c|c|c|c|c|c|}
        \hline
        \rule{0pt}{10pt}
        $N_v$ & $e$ & Order & $De$ & Order & $De_2^T$ & Order & $De_{2}^{\mathcal{M}}$ & Order \\
        \hline
        $200$   & $2.028$     & --     & $6.893e{+0}$     & --     & $1.934e{-1}$ & --     & $3.292$       & --     \\ \hline
        $800$   & $5.262e{-1}$ & $1.946$ & $3.504e{+0}$     & $0.976$ & $5.267e{-2}$ & $1.877$ & $8.798e{-1}$   & $1.904$ \\ \hline
        $3200$  & $1.328e{-1}$ & $1.987$ & $1.705e{+0}$     & $0.990$ & $1.061e{-2}$ & $2.312$ & $2.251e{-1}$   & $1.967$ \\ \hline
        $12800$ & $3.349e{-2}$ & $1.987$ & $8.543e{-1}$ & $0.997$ & $2.816e{-3}$ & $1.913$ & $5.668e{-2}$   & $1.990$ \\ \hline
    \end{tabular}}
    \label{tab:Torus_results}
\end{table}

The numerical results are reported in Table~\ref{tab:Torus_results}. As observed from the table, the $L^2$ error exhibits optimal convergence of order $\mathcal{O}(h^2 + \tau^2)$, while the $H^1$ error converges at the expected rate of $\mathcal{O}(h + \tau)$. Regarding the recovered gradients, we observe superconvergent behavior of order $\mathcal{O}(h^2 + \tau^2)$ for both $G_{\tau}u_h$ and $G_hu_h$, which is consistent with the theoretical results established in Theorem~\ref{thm:time_superconvergence} and Theorem~\ref{thm:space_superconvergence}.

\subsection{Numerical example 2}
In this example, we consider the model equation \eqref{equ:model} on a more general surface, following the settings in \cite{dziuk2013review, dong2020pppr}. The exact surface $\mathcal{M}$ is implicitly defined as the zero level set of the function
\begin{equation}
    \varphi(x)=\frac{1}{4}x_1^2+x_2^2+\frac{4x_3^2}{\left(1+\frac{1}{2}\sin(\pi x_1)\right)^2}-1.
\end{equation}
The source term $f$ and the Neumann boundary conditions $\mu_i$ ($i = 0, 1$) are prescribed such that the exact solution is given by $u = x_1 x_2 e^t$.

\begin{table}[H]
    \caption{Numerical results of time-space Laplace-Beltrami type equation for Numerical Example 2.}
    \centering
    		\resizebox{\textwidth}{!}{
    \begin{tabular}{|c|c|c|c|c|c|c|c|c|}
        \hline
        \rule{0pt}{10pt}
        $N_v$ & $e$ & Order & $De$ & Order & $De_2^T$ & Order & $De_{2}^{\mathcal{M}}$ & Order \\
        \hline
        $1153$  & $1.375e{-1}$ & --      & $1.366$     & --      & $6.221e{-2}$ & --      & $1.138$       & --     \\  \hline
        $4606$  & $5.864e{-2}$ & $1.229$ & $7.377e{-1}$ & $0.889$ & $2.591e{-2}$ & $1.264$ & $4.621e{-1}$  & $1.300$ \\  \hline
        $18418$ & $1.447e{-2}$ & $2.019$ & $3.741e{-1}$ & $0.980$ & $6.678e{-3}$ & $1.956$ & $1.571e{-1}$  & $1.557$ \\ \hline
        $73666$ & $3.936e{-3}$ & $1.878$ & $1.887e{-1}$ & $0.988$ & $1.799e{-3}$ & $1.893$ & $4.830e{-2}$  & $1.702$ \\ \hline
    \end{tabular}}
        \label{tab:Elliott_results}
\end{table}

The numerical errors with respect to the number of vertices $N_v$ are reported in Table~\ref{tab:Elliott_results}. From the table, it is evident that the $L^2$ error $e$ converges at a second-order rate, while the $H^1$ error $De$ exhibits first-order convergence. Moreover, superconvergent behavior is observed in the recovered time derivative $De_2^T$, with a rate of approximately $\mathcal{O}(h^{1.9} + \tau^{1.9})$, and in the recovered spatial gradient $De_2^{\mathcal{M}}$, with a rate of approximately $\mathcal{O}(h^{1.9} + \tau^{1.7})$. These findings are in agreement with our theoretical predictions.

\subsection{Numerical example 3}
In this example, we examine the model equation \eqref{equ:model} on a general surface defined implicitly by the level set function
\begin{equation}
    \varphi(x)=(x_1-x_3^2)^2+x_2^2+x_3^2-1.
\end{equation}
The exact solution is taken to be the same as in the previous example. As before, the source term and Neumann boundary data are derived consistently from the exact solution.

\begin{table}[H]   
    \centering
        \caption{Numerical results of time-space Laplace-Beltrami type equation for Numerical Example 3.}
 		\resizebox{\textwidth}{!}{
    \begin{tabular}{|c|c|c|c|c|c|c|c|c|}
        \hline
        \rule{0pt}{10pt}
        $N_v$ & $e$ & Order & $De$ & Order & $De_2^T$ & Order & $De_{2}^{\mathcal{M}}$ & Order \\
        \hline
        $243$    & $1.883e{-1}$ & --      & $1.341$     & --      & $5.840e{-2}$ & --      & $9.952e{-1}$  & --      \\ \hline
        $966$    & $5.412e{-2}$ & $1.799$ & $7.174e{-1}$ & $0.901$ & $1.923e{-2}$ & $1.602$ & $3.720e{-1}$  & $1.420$ \\ \hline
        $3858$   & $1.440e{-2}$ & $1.910$ & $3.664e{-1}$ & $0.969$ & $5.182e{-3}$ & $1.892$ & $1.178e{-1}$  & $1.659$ \\ \hline
        $15426$  & $3.697e{-3}$ & $1.962$ & $1.843e{-1}$ & $0.991$ & $1.332e{-3}$ & $1.961$ & $3.301e{-2}$  & $1.836$ \\ \hline
    \end{tabular}}
     \label{tab:Dziuk_results}
\end{table}

The error histories are summarized in Table~\ref{tab:Dziuk_results}. As in the previous example, optimal convergence rates are observed for both the $L^2$ error $e$ and the $H^1$ error $De$. For the recovered time derivative, a superconvergent rate of $\mathcal{O}(h^2)$ is achieved, while for the recovered surface gradient, a superconvergent rate of approximately $\mathcal{O}(h^{1.8})$ is observed.

\section{Conclusion}\label{sec:con}
We have studied the discretization of Laplace–Beltrami type equations on time-space manifolds. This class of equations arises from the solution process of dynamic optimal transport on general manifolds. We proposed a coupled FDM–sFEM method. By employing the PPR in the temporal interval and the PPPR on the spatial surface, we established the superconvergence theory of gradient recovery on product manifolds, which provides theoretical support for the gradient-enhanced ADMM algorithms proposed in \cite{dong2025ot}. In addition, we present a new geometric error analysis framework based on directly approximating the Riemannian metric. This framework is applicable in a more general setting where the discrete surface is not required to interpolate the exact surface, and exact geometric information—such as normal vectors—is not needed. Several numerical examples are provided to confirm our theoretical results. Moreover, the manifold $(0,1)\times \mathcal{M}$ represents a special case of the product manifolds $\mathcal{M}_1 \times \mathcal{M}_2$, on which the numerical analysis of PDEs will be of interest for future investigation.

\section*{Acknowledgments} 
The work of GD was partially supported by the National Natural Science Foundation of China (NSFC) grant No. 12471402, and the NSF of Hunan Province grant No. 2024JJ5413. The work of HG was partially supported by the Andrew Sisson Fund and the Faculty Science Researcher Development Grant of The University of Melbourne. The work of CJ and ZS was supported by the NSFC grant No. 92370125.

\section*{Data availability}
Code implemented in the paper to reproduce the numerical results will be available upon request.

\bibliographystyle{plainurl} 
\bibliography{references}

\end{document}